\documentclass[preprint,11pt]{elsarticle}
\usepackage{amsfonts}
\usepackage{eurosym}
\usepackage{amsmath,amssymb,amsthm}
\usepackage{fullpage}
\usepackage[utf8]{inputenc}
\usepackage[T1]{fontenc}

\journal{Journal of Differential Equations}

\usepackage{amsmath,amsthm,amsfonts,amssymb}
\usepackage[mathscr]{euscript}
\usepackage{hyperref} 
\usepackage{mathrsfs} 
\usepackage{graphicx}
\usepackage{color} 
\usepackage{epsfig}

\newtheorem{lemma}{Lemma}[section]
\newtheorem{theorem}{Theorem}[section]
\newtheorem{corollary}{Corollary}[section]
\newtheorem{proposition}{Proposition}[section]
\theoremstyle{definition}
\newtheorem{definition}{Definition}
\theoremstyle{definition}
\newtheorem{example}{Example}
\theoremstyle{definition}
\newtheorem{remark}{Remark}

\numberwithin{equation}{section}

\begin{document}

\begin{frontmatter}
 
\title{On exponential stability of generalized linear functional differential equations}

\author[label1]{Leonid Berezansky}
\author[label2]{Elena Braverman}

\address[label1]{Dept. of Math., Ben-Gurion University of the Negev,
Beer-Sheva 84105, Israel}

\address[label2]{Dept. of Math. and Stats., University of
Calgary,2500 University Drive N.W., Calgary, AB  T2N 1N4, Canada}

\begin{abstract}
We develop a general approach to uniform exponential stability for a broad class of linear functional differential equations with Volterra operators.
The considered class includes delay differential equations, integro-differential equations, 
equations with distributed delays, vector equations, 
and systems of scalar functional differential equations. 
Standard assumptions on the equation 
are relaxed to include  measurable parameters.
Some results remain valid for equations 
with  unbounded delays and oscillatory coefficients and kernels.
For all the above mentioned classes of equations, explicit uniform exponential stability tests are obtained.
\end{abstract}

\begin{keyword}
Abstract Voltera operators; functional differential equations; 
uniform exponential stability; matrix measure; 
unbounded delays; oscillatory  coefficients

{\bf AMS Subject Classification:}  34K20, 34K25

\end{keyword}

\end{frontmatter}

\section{Introduction}
\label{int}

Functional differential equations (FDEs) continue  to attract considerable attention, and an extensive body of literature has been devoted to the stability analysis  of  FDEs, including scalar and vector delay differential equations (DDEs) \cite{BB2006a,BDSZ2,BDSZ3,DomGus,Fridman,GH2001,GP,Krisztin,LizPituk,Malygina,Yoneama}, systems of scalar  
FDEs \cite{AppGyori,BDSZ1,Faria2,Faria-2022,Ngoc2016},
 integro-differential equations (IDEs) \cite{BB2025,Kolm,Ngoc2017,Tunc2016,TuncTunc},
differential equations with distributed delays %(DDDEs) 
\cite{BB2024,BTT,BZ,Gil2}, and neutral functional differential equations \cite{BB2,BB1,JFI2023,Dom1,Ngoc1,Ngoc2,Sun1,Park}.
Stability results for these classes of FDEs can be found in the monographs \cite{BBbook,AMR,AS,Burton,Cord,Gil3,Gil,
Gopalsamy,Hale,KM,KN,Kuang}. 

In some of these publications, stability conditions are derived using operators of the form $F(x_t)$, where $x_t = x(s)$ for $s \le t$.
However, such operators act only in the spaces of continuous functions and consequently are not applicable to equations with measurable parameters
and potentially discontinuous initial conditions.

The study of the asymptotic properties of differential equations with distributed delays originates from the classical work of Vito Volterra \cite{Volterra}.
In this  paper, we establish
% derive 
stability criteria for linear functional differential equations with abstract Volterra operators acting in the spaces of both continuous and measurable functions.
This unified framework allows the same explicit exponential stability conditions to be applied directly to scalar equations,
vector models and systems of scalar functional differential equations,
where a model may involve terms with different types of delays.
%Consequently, 
The proposed approach generalizes existing stability results
and yields new criteria for a broad class of equations, which we later illustrate.
Throughout the paper, we consider only \emph{linear} functional differential equations.
Necessary definitions and auxiliary results are given in Section~\ref{Prel}.

Our method reduces exponential stability % problem 
to boundedness %the existence 
of solutions to an operator equation
in spaces of measurable functions on the half-line.
It is based on the Bohl-Perron theorem for abstract Volterra operators,
first introduced in \cite{ABSC}, see also \cite{AS} for its detailed outline.
The Bohl-Perron theorem establishes a dichotomy: either an equation is exponentially stable 
or it has an unbounded solution for some bounded right-hand side on a half-line.
So far, explicit exponential stability conditions obtained via this theorem
have been available only for specific classes of functional differential equations.
Here, we develop this approach for an abstract equation with Volterra operators
and then apply it to specific classes of FDEs.
In particular, we obtain stability tests for all classes of equations listed above,
as well as for equations with unbounded delays,
equations with both concentrated and distributed delay types, and  FDEs %functional differential equations
with oscillatory coefficients and kernels. The latter case is known
to be especially challenging even for scalar equations with a single delay.

%The main objective of this paper is the following.
In the current paper, rather than proving separate stability criteria for DDEs, equations with a distributed delay, IDEs and systems of DDEs,
%\begin{itemize}
%\item DDEs,
%\item distributed-delay equations,
%\item IDEs,
%\item systems,
%\end{itemize}
we derive one abstract theorem and recover all these cases as corollaries.
%To the best of our knowledge, this is the first 
We are not aware of any abstract exponential stability criterion based on Volterra operators 
that simultaneously applies to all the above classes of equations
for models with  measurable parameters in the existing literature. 
Unlike many known results, we consider discontinuous coefficients and unbounded delays, 
without a specific prescribed delay structure.

Most of currently available stability results are obtained for either continuous parameters or a non-oscillatory ``leading" term. 
In Example~\ref{ex7.1} of Section~\ref{sec:6},  we consider a scalar delay differential equation
with two delays, where both coefficients and their sum are oscillatory functions.
We are not aware of any previous results capable of establishing exponential stability for such an equation.
Example~\ref{ex7.4} 
demonstrates a system for which existing criteria are inconclusive whereas the present theorem establishes exponential stability.

The structure of the paper is as follows. After introducing relevant background and notations, we prove auxiliary results in Section~\ref{Prel},
which are used in justification of main Theorem~\ref{th3.1} in Section~\ref{sec:3}. 
In Section~\ref{sec:4}, Theorem~\ref{th3.1}  is applied to particular classes of systems with Volterra operators with one, two, and several terms, one of which can be non-delayed.
Section~\ref{sec:5} investigates equations and systems with concentrated delays, distributed delays, and integro-differential terms, as well as models involving different types of delays, outlining exponential stability condition in each case.
The obtained results are illustrated with examples in Section~\ref{sec:6} and discussed in Section~\ref{sec:7}.

%\begin{remark}\label{r1.1}
%All equations in the paper we consider with an arbitrary initial point $t_0$. However, 
%all the conditions for exponential stability,
%formulated below, must be satisfied for sufficiently large values $t$. 
%Therefore, we have to consider a family of equations
%defined for $t\geq t_1\geq t_0$, as well as the initial value problems for such equations. Considering such a family is also necessary for obtaining conditions for uniform exponential stability (UES) of all equations, considering here, including equations with abstract Volterra operators.
%\end{remark}

\section{Preliminaries}
\label{Prel}

\subsection{Notations and auxiliary results}

By ${\|x\|}:={\|x\|}_{\infty}=\max_{i=1,\dots n}|x_i|$  we denote the maximum norm of an $n$-dimensional column vector $x=(x_1,\dots,x_n)^T$ in the %finite n-dimensional
 vector space ${\mathbb R}^n$, 
$(\cdot)^T$ is transposition,
the same symbol $\| A \|$ stands for the induced matrix norm $\displaystyle {\| A \|}:={\| A\|}_{\infty}=\max_{i=1,\dots n}\sum_{j=1}^n|a_{ij}|$ of a matrix $A=\{a_{ij}\}_{i,j=1}^n$. We use the same symbol $0$ for the zero constant, the null-vector, and the null-matrix, $E$ is the the identity $n\times n$ matrix, ${\mathbb R}^+ =[0,+\infty)$.
We will apply further the logarithmic  matrix norm  (the matrix measure) defined  \cite{Z} as 
%\begin{equation*}
%\label{2.1}
$\displaystyle \mu_A:=\lim_{\delta\rightarrow 0^+}{\delta}^{-1}\big( \|E+\delta A\|-1 \big)$.
%\end{equation*}
Recall that for the maximum  norm  the matrix measure \cite[Table 3.1, p.\,286]{KM}  is
$\displaystyle 
\mu_A =\max_{i=1\dots,n}\left\{a_{ii}+\sum%\nolimits
_{j=1,j\neq i}^n|a_{ij}|\right\}.
$
We will use the notation ${{\bf L}}_{\infty}[t_0,\infty)$ for the space of all essentially bounded $n$-dimensional real functions
on $[t_0,\infty)$ with
%For $y\in {{\bf L}}_{\infty}[t_0,\infty)$, define 
the essential supremum norm $\displaystyle {\|y\|}_{[t_0,\infty)}:= \mbox{ess}\sup_{t\in [t_0,\infty)} \|y(t)\|$.
Denote by  ${\bf C}[t_0,\infty)$  the space  of all continuous   $n$-dimensional bounded % on $[t_0,\infty)$
 functions  $x:[t_0,\infty) \to {\mathbb R}^n$ with the same supremum norm as in ${{\bf L}}_{\infty}[t_0,\infty)$.
Similarly, we introduce  the spaces  ${{\bf L}}_{\infty}[t_0,b]$ and ${\bf C}[t_0,b]$ on a finite interval with an appropriate sup-norm.
If a matrix function $A(t)$ is defined on $[t_0,\infty)$  and $\|A(t)\|$ is the induced matrix norm,
denote %the norm 
$\displaystyle  {\|A\|}_{[t_0,\infty)}:= \mbox{ess} \sup_{t\in [t_0,\infty)}\|A(t)\|$, we use the same notation for scalar functions.

In the proofs of the main statements, we apply three lemmas. 
The first one is a slight modification of  \cite[Lemma~2]{BDSZ1},
the second one is a part of \cite[Theorem~2.2.2]{G},
and the third lemma recalls the Coppel inequality~\cite[p.~9]{Coppel}
for linear systems of ordinary differential equations (ODEs).

\begin{lemma}\label{lemma2.1}  \cite[Lemma~2]{BDSZ1}
For any Lebesgue measurable $a\colon \! [t_0,\infty)\to {\mathbb R}^+$ 
and
$\omega\in {{\bf L}}_{\infty}[t_0,\infty)$, 
$$
\left|\int_{t_0}^t e^{-\int_s^t a(\tau)d\tau}a(s)\omega(s)ds\right|
\leq {\mathrm{ess}} \sup_{s\geq t_0}|\omega(s)|,\,\,\,\,
t\geq t_0.
$$
\end{lemma}

\begin{lemma}\label{lemma2.1a}  \cite[Theorem~2.2.2~(vi), p. 39]{G}
If $f\colon \! [t_0,\infty)\to {\mathbb R}^+$
is Lebesgue measurable and 
\begin{equation}\label{2.2a}
\lim\limits_{t\rightarrow\infty}\int_t^{t+T}f(s)\, ds=0
\end{equation} 
for some $T>0$, then for any $\lambda>0$ we have  $\displaystyle \lim_{t\rightarrow\infty}\int_{t_0}^t e^{-\lambda(t-s)} f(s)\, ds=0$.
\end{lemma}
In particular, equality  \eqref{2.2a} holds if $\int_{t_0}^{\infty} f(s)ds<\infty$ or $\lim_{t\rightarrow\infty}f(t)=0$.

Denote  by $X_A(t,s)$, $t,s\in [t_0,\infty)\times [t_0,\infty)$,  the fundamental matrix of the
system
\begin{equation}\label{2.2}
\dot{x}(t)=A(t)x(t),\,\,\,\,\,t\in [t_0,\infty),
\end{equation}
which is the solution of matrix equation \eqref{2.2} with an initial point $s\geq t_0$ and the initial condition $X_A(s,s)=E$.
We will say that \eqref{2.2} is {\bf uniformly exponentially stable (UES)}  if
for some constants $M>0, \lambda>0$, the fundamental matrix satisfies
\begin{equation}\label{2.3c}
\|X_A(t,s)\|\leq Me^{-\lambda (t-s)}, ~t_0\leq s\leq t,
\end{equation}
where $\lambda$ and $M$ do not depend on $s$. If $M=M(s)$ then \eqref{2.2} is {\bf  exponentially stable (ES)}.

If $A(t)=a(t)$ is a scalar Lebesgue measurable essentially bounded on $ [t_0,\infty)$ function, then equation \eqref{2.2} is UES if and only if  there exists a constant $H>0$
such that 
\begin{equation}\label{2.2b}
\limsup_{t\rightarrow\infty}\int_{t}^{t+H}a(s)ds<0.
\end{equation}

\begin{remark}\label{r.add1}
If $A(t)$ is an unbounded function,  equation \eqref{2.2} may be ES but not UES. 
As an example, consider a scalar $A(t)=a(t)=-0.5 \lfloor t \rfloor-   \lfloor t \rfloor \sin(2\pi t)$, where $ \lfloor t \rfloor$ is an integer part of $t$.
\end{remark}

\begin{lemma}\label{lemma2.2}  \cite[p.~9]{Coppel}
The fundamental matrix function $X_A(t,s)$ of   \eqref{2.2}   satisfies
\begin{equation*} %\label{2.3}
\left\|X_A(t,s)\right\|\leq X_{\mu_A}(t,s)=\exp \left({\int_s^t \mu_A (\xi)  d\xi}\right),\,\,\,\,\,t>s\ge t_0,
\end{equation*}
where $ X_{\mu_A}(t,s)$ is the fundamental function of the scalar ODE %ordinary differential equation (ODE)
%\begin{equation}\label{2.3a}
$\dot{x}(t)=\mu_A(t)x(t)$.
%\end{equation}
If
%\begin{equation*}  %\label{2.3b}
$\displaystyle  |X_{\mu_A}(t,s)|\leq Me^{-\lambda (t-s)}$ for $t_0\leq s\leq t$
%\end{equation*} 
and some constants $M>0, \lambda>0$,  estimate \eqref{2.3c} with the same constants
 holds for the fundamental function of equation \eqref{2.2}.
\end{lemma}

\subsection{Volterra operators}
\label{subsec:Volterra}

\begin{definition}\label{d2.2}
Let $X$ be a linear space of vector functions $x:[t_0,\infty) \to {\mathbb R}^n$  and $V: X \to X$ be an operator. 
If for every $t_1>t_0$ and two functions $x(t), y(t)$, the equality $x(t)=y(t)$ for $t\in [t_0,t_1]$ implies
$(Vx)(t)=(Vy)(t)$ for  $t\leq t_1$ then $V$ is called {\bf a Volterra operator}.
\end{definition}

Main examples of Volterra operators and  conditions on the parameters are listed together  below:
\begin{description}
\item[an integral Volterra operator]
\begin{equation}\label{2.4}
(Vx)(t)=\int_{t_0}^t K(t,s)x(s)ds, \quad t\geq t_0,
\end{equation}
%where 
%\\
$(i)$ the entries $k_{ij}\colon [t_0,\infty)\times [t_0,\infty)\to\mathbb{R}$,
$i,j=1,\dots,n$ %, $k=1,\dots,l$ 
of the matrix $K$  %in~\eqref{2.4} 
are Lebesgue measurable
essentially bounded functions on $[t_0,\infty)\times [t_0,\infty)$, and
$\displaystyle {\mathrm{ess}} \sup_{t\geq t_0} \int_{t_0}^t \|K(t,s)\|ds<\infty$;

\item[an operator with a concentrated delay]
\begin{equation}\label{2.5}
(Vx)(t)=A(t)x(h(t)),\quad  h(t)\leq t, \quad t\geq t_0; \quad  x(t)=\varphi(t), \quad  t\leq t_0,
\end{equation}
%where \\
$(ii)$ % the delay 
$h\colon [t_0,\infty)\to\mathbb{R}$ is Lebesgue measurable locally essentially bounded,
$h(t)\leq t$, $t\ge t_0$ and $\lim_{t\rightarrow\infty}h(t)=\infty$;
\\
$(iii)$ the entries $a_{ij}\colon [t_0,\infty)\to\mathbb{R}$,
$i,j=1,\dots,n$,  of the matrix $A$ are  Lebesgue measurable essentially bounded %functions 
on $[t_0,\infty)$;

\item[an integral Volterra operator with delay]
\begin{equation}\label{2.6}
(Vx)(t)=\int_{h(t)}^t K(t,s)x(s)ds, \quad t\geq t_0; \quad  x(t)=\varphi(t),  \quad  t\leq t_0
\end{equation}
with conditions $(i)$ and $(ii)$, where in (i)  the lower limit of the integral is replaced with $h(t)$;

\item[an operator with a distributed delay]
\begin{equation}\label{2.7}
(Vx)(t)=\int_{h(t)}^t d_s G(t,s)x(s), \quad t\geq t_0; \quad  x(t)=\varphi(t), \quad  t\leq t_0,
\end{equation}
where condition $(ii)$ holds and
\\
$(iv)$  all the entries $g_{ij}$ of the $n\times n$ matrix $G(t,\cdot)$ are left continuous
scalar functions of bounded variation and for each $s$
the variations 
%${\rm Var}$ on the segment $[t_0,s]$
%$
$\displaystyle 
p_{ij}(t,s) = {\rm Var}_{\tau \in [t_0,s]} g_{ij}(t,\tau)$, $i,j=1, \dots, n
$
on %the segment 
$[t_0,s]$ 
are  essentially bounded  in $[t_0,\infty)\times [t_0,\infty)$. 
The variation  ${\rm Var}_{[a,b]} A(t)$ is  defined as %a scalar or vector function $A(t)$:
$$ {\rm Var}_{[a,b]} A(t) = \sup_{\Omega} \sum_{j=1}^k \| A(t_{j}) - A(t_{j-1} ) \|, \quad \Omega = \{t_0,t_1, \dots,t_k\}, ~~a=t_0 < t_1 < \dots < t_k = b, $$ 
where the supremum is taken over all the partitions $\Omega$ of the segment $[a,b]$.
To avoid confusion, the variable in which the variation is taken will sometimes be indicated. 
%By the variation of the matrix function $A(t)$ we mean a scalar function defined as
%If $A(t)$ is a differentiable function then $ Var_{[a,b]} A(t) \leq \sup_{t\in[a,b]}\|A'(t)\|(b-a)$.
\end{description}

For the initial function $\varphi$ we also assume that, for \eqref{2.5}-\eqref{2.6},
\\
$(v)$ $\varphi :(-\infty,t_0)\rightarrow {\mathbb R}^n$ is a bounded Borel measurable vector  function.

For \eqref{2.7}, we  impose a condition: %assume that 
\\
$(vi)$  $\varphi :(-\infty,t_0)\rightarrow {\mathbb R}^n$  belongs to a class of functions  guaranteeing existence of the Lebesgue-Stiltjes integrals
(for any $G$, this is guaranteed for bounded continuous functions). 

From all the operators defined above, generally, only the integral Volterra operator in \eqref{2.4} is a linear operator.
If the initial function $\varphi (t)\equiv 0$ then all the operators in \eqref{2.4}-\eqref{2.7}  are linear.
With a slight modification of  \eqref{2.5}-\eqref{2.7}, we get linear operators:

\begin{description}
\item[a linear  operator with concentrated delay]
\begin{equation}\label{2.8}
(Vx)(t)=\left\{\begin{array}{cc}
A(t)x(h(t)),&  h(t)\geq t_0,\\
0,& h(t)< t_0;
\end{array}\right.
\end{equation}

\item[a linear integral Volterra operator with delay]
\begin{equation}\label{2.9}
(Vx)(t)=\int_{\max\{t_0, h(t)\}}^t K(t,s)x(s)ds, \quad t\geq t_0;
\end{equation}

\item[a linear operator with a distributed delay]
\begin{equation}\label{2.10}
(Vx)(t)=\int_{\max\{t_0, h(t)\}}^t d_s G(t,s)x(s), \quad t\geq t_0.
\end{equation}
\end{description}

If $V_1, V_2$ are linear Volterra operators and $\alpha, \beta$ are scalars,  their linear combination  $\alpha V_1+\beta V_2$ 
is also a linear Volterra operator.

\begin{definition}\label{d2.3}
Assume that $V$ is a linear Volterra operator, $\tau(t)\geq 0$ for  $t\geq t_0$ is a scalar locally essentially bounded function such that
 the value of $(Vx)(t)$ depends only on the values of $x(s)$ for $t-\tau(t)\leq s\leq t$, and 
$\tau(t)$ is the infimum of $\tau_1(t)$ such that $(Vx)(t)$ depends only on the values of $x(s)$, $s \in [ t-\tau_1(t),t]$.
%for any $\tau_1(t)$ with this property
% $\tau)t)\leq \tau_1(t)$.
Then  the Volterra operator $V$ {\bf has memory $\mathbf \tau(t)\geq 0$}. The function
$h(t)=t-\tau(t)\leq t$ is the delay function of operator $V$.
The memory $\tau(t)$ is {\bf finite} (the equation has a {\bf bounded delay})  
if $\tau(t)\leq \tau$ for some number $\tau>0$ and any $t \geq t_0$.
We will also say that the delay $h(t)$ {\bf has (A)-property} if
for any $t_1\geq t_0$ there exists $t_2>t_1$ such that $h(t)\geq t_1$ for $t\geq t_2$.
\end{definition}

Condition $(ii)$ implies that  the delay $h(t)$ has (A)-property.
For example, unbounded delay functions $h(t)=\lambda t$, $\lambda \in (0,1)$, $ t\geq 0$ and $h(t)=t-\ln t$,  $t\geq 1$ have (A)-property. 
For a sum of the operators  $V=V_1+V_2$ with the memories $\tau_1(t)$, $\tau_2(t)$, respectively, the memory of $V$ does not exceed
%the earliest reference point which is not less than 
$\tau(t)=\max \{\tau_1(t),\tau_2(t)\}$.
The operators defined by \eqref{2.8}-\eqref{2.10}  have memory $\tau(t)=t-h(t)$.

\begin{lemma}\label{lemma2.3}
For the  operators $V$  defined by \eqref{2.8}-\eqref{2.10},  we have %the following inequality holds
\begin{equation}\label{2.11}
\|(Vx)(t)\|\leq b(t)\sup_{\max\{t_0, h(t)\}\leq s\leq t}\|x(s)\|.
\end{equation}
Here 
$b(t)=\|A(t)\|$ for $V$ as in %the operator defined by 
\eqref{2.8},
$\displaystyle b(t)=\int\limits_{\max\{t_0, h(t)\}}^t \|K(t,s)  \|ds$ for $V$ as in \eqref{2.9}, and
$\displaystyle  b(t)={\rm Var}_{\zeta  \in [\max\{t_0, h(t)\}, t]} G(t,\zeta)$  for $V$ as in \eqref{2.10}.
%\vspace{3mm}
If the operator $V$ is  defined by \eqref{2.4}, inequality \eqref{2.11} holds with $h(t)\equiv t_0$ and 
$\displaystyle b(t)=\int\limits_{t_0}^t \|K(t,s)  \|ds$.
\end{lemma}

The proof of the lemma is elementary and will be omitted.

\subsection{The $\Delta $-property}

The following definition will be used to explore exponential stability. 
\begin{definition} \cite[p.~106]{AMR} \label{d2.4}
Assume  that 
%for any $t_1>t_1$, 
$V\colon {\bf{C}} [t_0,\infty)\to {\bf L}_{\infty}[t_0,\infty)$ is 
a linear bounded Volterra operator %$V:C[t_0,\infty)\rightarrow L_{infty}[t_0,\infty)$ 
and there exists a constant
 $\lambda>0$ such that the inequality 
$\|y(t)\|\leq e^{-\lambda t}$, $t\geq t_0$ implies 
$\|(Vy)(t)\|\leq M e^{-\lambda t}$, $t \in [t_0,\infty)$ for some $M>0$.
Then the operator $V$ has {\bf a $\mathbf \Delta$-property}. 
\end{definition}

It is obvious that the sum of operators with a $\Delta$-property also has a $\Delta$-property.

\begin{example}\label{ex2.1}
Consider the operators $V_1,V_2$ on the space ${\bf L}_{\infty}[t_0,\infty)$ of scalar functions, where
$$
(V_1y)(t)=\left\{
\begin{array}{ll}
y(0.5 t), &t\geq 0,\\0,& t<0,
\end{array}\right.  \quad 
(V_2y)(t)=\left\{
\begin{array}{ll}
e^{-2t}y(0.5 t), &t\geq 0,\\0,& t<0.
\end{array}\right.
$$
The operator $V_1$ has no $\Delta$-property, while the operator $V_2$ has  a $\Delta$-property with any $\lambda\in(0,4]$.
\end{example}

\begin{lemma}\label{lemma2.4}
The operator $V$ in  \eqref{2.8} has a $\Delta$-property under any of the two conditions:

a) there exists $\lambda>0$ such that
\begin{equation}\label{2.12}
\sup_{t\geq t_0}e^{\lambda(t-h(t))}\|A(t)\|<\infty;
\end{equation}

b) for some  $\tau>0$, 
%\begin{equation*}  %\label{2.13}
$t-h(t)\leq \tau$ and $\displaystyle \sup_{t\geq t_0}\|A(t)\|<\infty$. 
%\end{equation*}
\end{lemma}

\begin{proof}
a)  Denote $z(t)=e^{ \lambda t } y(t)$ and let $\|y(t)\|\leq e^{ -\lambda t }$ for  $t\geq t_0$, then $\|z(t)\|\leq 1$, $t\geq t_0$,
and % We have
\begin{align*}
& \displaystyle  \left\| e^{ \lambda t } (Vy)(t) \right\|= \left\|  e^{\lambda t} A(t)y(h(t)) \right\|
= \left\| A(t) \left[  e^{\lambda h(t)}   y(h(t)) \right]  e^{\lambda (t-h(t))}  \right\|
=   \displaystyle  \left\| A(t)z(h(t)) e^{\lambda (t-h(t))}   \right\|  \\ \leq  & e^{\lambda (t-h(t))} \|A(t)\| \, \|z(h(t))\|
\leq  M:=   \sup_{t\geq t_0} \left[ e^{\lambda (t-h(t))} \|A(t)\|  \right],
\end{align*}
since $\|z(h(t))\|\leq 1$ if $h(t)\geq t_0$ and $z(h(t))=0$ if $h(t)<t_0$.
Thus
$\displaystyle 
\left\|  e^{ \lambda t } (Vy)(t) \right\| =  e^{ \lambda t}  \|   (Vy)(t)\| \leq M$  and
%$\displaystyle  \leq M:=   \sup_{t\geq t_0} \left[ e^{\lambda (t-h(t))} \|A(t)\|  \right]$ . Thus 
$\|(Vy)(t)\|\leq M e^{-\lambda t}$.

b)  follows from a), as \eqref{2.12} is satisfied for any $\lambda >0$.
\end{proof}

\begin{lemma}\label{lemma2.5}
If there exists $\lambda>0$ such that either
\begin{equation*} %\label{2.14a}
\sup_{t\geq t_0}\int\limits_{\max\{t_0, h(t)\}}^t   \!\! e^{\lambda (t-s)}\|K(t,s)\|ds <\infty
%\end{equation*}
%\begin{equation*} %\label{2.15a}
\mbox{\rm ~~   or ~~ } \sup_{t\geq t_0} \left\{ e^{\lambda(t-h(t))}  \!\!  \int\limits_{\max\{t_0, h(t)\}}^t  \!\!  \|K(t,s)\|ds \right\} <\infty
\end{equation*}
holds  then the operator $V$  in \eqref{2.9} has  a $\Delta$-property. 
\end{lemma}

\begin{proof}
%a) 
Again, let $z(t)=e^{\lambda t}y(t)$ and  $\|y(t)\|\leq e^{-\lambda t}$, $t\geq t_0$, then $\|z(t)\|\leq 1$ for $t\geq t_0$
and
$$
e^{\lambda t} \int\limits_{\max\{t_0, h(t)\}}^t K(t,s)y(s)ds
=\int\limits_{\max\{t_0, h(t)\}}^t K(t,s)   e^{\lambda s}y(s) ~e^{\lambda (t-s)}\, ds.
$$
Hence, the first inequality in the statement of Lemma~\ref{lemma2.5} leads to 
$$
%\begin{align*}
e^{\lambda t}  \| (Vy)(t)\|  %& 
\leq  \!\!\! \int\limits_{\max\{t_0, h(t)\}}^t   \!\!\!\!\!\!   e^{\lambda (t-s)} \|K(t,s)\| \|z(s ) \| ds
%\\ & 
\leq M:= \sup_{t\geq t_0}\int\limits_{\max\{t_0, h(t)\}}^t   \!\!\!\!\!\!   e^{\lambda (t-s)}\|K(t,s)\|ds <\infty,
$$
%\end{align*}
and by  definition, the operator $V$ has a $\Delta$-property.
Once the second inequality is satisfied, since
%b) Since 
%\begin{align*}
%\sup_{t\geq t_0}   \!\!\!  
$$ \int\limits_{\max\{t_0, h(t)\}}^t \!\!\! \!\!\!   e^{\lambda (t-s)}\|K(t,s)\|ds
\leq  % \sup_{t\geq t_0}
\int\limits_{\max\{t_0, h(t)\}}^t  \!\!\! \!\!\!   e^{\lambda (t-h(t))}\|K(t,s)\|ds
%\\
\leq  
%& 
\sup_{t\geq t_0}  \left\{ 
e^{\lambda(t-h(t))}  \!\!\!  \!\!\!  \int\limits_{\max\{t_0, h(t)\}}^t  \!\!\!  \!\!\!   \|K(t,s)\|ds \right\} , %<\infty ,
$$
%\end{align*}
%Part a) yields 
the first inequality in the statement %of Lemma~\ref{lemma2.5} 
also holds, and thus the operator $V$ has a $\Delta$-property.
\end{proof}

\begin{lemma}\label{lemma2.6}
If there exists $\lambda>0$ such that  anyone of the inequalities
\begin{equation}\label{2.16}
\sup_{t\geq t_0}\int_{\max\{t_0, h(t)\}}^t e^{\lambda (t-s)}  {\rm  Var}_{\zeta \in [\max\{t_0, h(s)\}, s]} G(t,\zeta) ds <\infty,
\end{equation}
%or
%b)
\begin{equation}\label{2.17}
\sup_{t\geq t_0}e^{\lambda(t-h(t))} {\rm  Var}_{\zeta \in [\max\{t_0, h(t)\}, t]}  G(t,\zeta)<\infty
\end{equation}
holds then the operator $V$ in  \eqref{2.10} has a $\Delta$-property. 
\end{lemma}

\begin{proof}
The proof when \eqref{2.16} holds follows the scheme of the proof of the first part  in Lemma~\ref{lemma2.5}.
If  \eqref{2.17} is satisfied,  the reference to the inequality  
$$
\sup_{t\geq t_0}\int_{\max\{t_0, h(t)\}}^t e^{\lambda (t-s)}  {\rm  Var}_{\zeta \in [\max\{t_0, h(s)\}, s]} G(t,\zeta) ds  \leq 
\sup_{t\geq t_0}\left\{ e^{\lambda(t-h(t))} {\rm  Var}_{\zeta  \in [\max\{t_0, h(t)\}, t]}  G(t,\zeta) \right\}
$$
implies a $\Delta$-property for $V$.
\end{proof}

\begin{lemma}\label{lemma2.7}
If there exists $\lambda>0$ such that either %at least one of the conditions
\begin{equation*} %\label{2.14}
\sup_{t\geq t_0}\int_{t_0}^t e^{\lambda (t-s)}\|K(t,s)\|ds <\infty \mbox{\rm ~ or ~~}
%\end{equation*}
%\begin{equation*} %\label{2.15}
\sup_{t\geq t_0}e^{\lambda(t-t_0)}\int_{t_0}^t\|K(t,s)\|ds<\infty
\end{equation*}
 holds  then the operator $V$ introduced in  \eqref{2.4} has a  $\Delta$-property. 
\end{lemma}
%The proof of this statement 
Lemma~\ref{lemma2.7} is just a corollary of Lemma~\ref{lemma2.5} for $h(t) \equiv t_0$.

\section{Exponential stability for differential equations with abstract Volterra operators}
\label{sec:3}

Everywhere in the paper, we assume that all  linear  Volterra operators 
$V:{\bf C}[t_0,\infty)\rightarrow  {\bf  L}_{\infty}[t_0,\infty)$ are bounded,  $\tau(t)$ is the memory 
and $h(t)=t-\tau(t)$ is the delay of  $V$.
%We could write that 
From the definition of the Volterra operator, we assume $x(t)=0$ for $t <t_0$.

%\begin{remark}\label{r1.1}
We consider all equations in the paper with an arbitrary initial point $t_0$. 
However, for uniform exponential stability we need to consider flexible 
initial points, for this purpose we assume initial conditions starting at any $t \geq t_1 \geq t_0$,
including equations with abstract Volterra operators.

%all the conditions for exponential stability,
%formulated below, must be satisfied for sufficiently large values $t$. 
%Therefore, we have to consider a family of equations
%defined for $t\geq t_1\geq t_0$, as well as the initial value problems for such equations. Considering such a family is also necessary for obtaining conditions for uniform exponential stability (UES) of all equations, considering here, including equations with abstract Volterra operators.
%\end{remark}

Consider now the equation with a Volterra operator $V$ 
\begin{equation}\label{2.19}
\dot{x}(t)=(Vx)(t)+f(t),\quad t\geq t_0,
\end{equation}
and an initial value problem (IVP) starting at  %from any initial point 
$t_1\geq t_0$, where \eqref{2.19} holds for $t \geq t_1$ and
\begin{equation}\label{2.19a}
x(t)=\varphi(t),\quad  t< t_1,  \quad x(t_1)=x_0.
\end{equation} 

Assume that the delay function $h(t)$ of operator $V$ has $A$-property.
For the initial function $\varphi$ in \eqref{2.19a} let either condition $(v)$ or $(vi)$ hold  and let  %let the right-hand side 
$f  \in {\bf  L}_{\infty}[t_1,b]$ for any $b >  t_1$. 
We assume that a solution of  problem \eqref{2.19}-\eqref{2.19a} exists and is unique. 
Such conditions for some specific  classes of FDEs are well known.
To define the fundamental function $X(t,s)$ of equation \eqref{2.19}, consider a family of operators
$(V_s x)(t):=(Vx_s)(t)$, $t\geq s\geq t_0$, where  %\\
$x_s(t)=\left\{\begin{array}{l}
0,~t<s, \\ x(t),~t\geq s.\end{array}\right.$
%We have 
Thus \\ $\|V_s x\|=\|Vx_s\|\leq \|V\| \|x_s\|\leq \|V\| \|x\|$ and
%leading to 
%Then 
$\|V_s\|_{[s,t_1]} \leq \|V\|_{[t_0,t_1]}$.

\begin{definition}\label{d2.0}
  {\bf The fundamental function}  $X(t,s)$ is a solution  of the matrix IVP
\begin{equation}\label{2.19aa}
\dot{x}(t)=(V_s x)(t), \quad t\geq s\geq t_0, \quad  x(t)=0,  \quad t<s,  \quad  x(s)=E.
\end{equation}
\end{definition}

\begin{theorem}\label{th2.1}
Denote  $\tilde{\varphi}(t)=\left\{\begin{array}{cc}
0, &t\geq t_1, \\ \varphi(t), &  t<t_1.\end{array}\right.$
The solution of IVP \eqref{2.19}-\eqref{2.19a} has a form
\begin{equation}\label{2.20a}
x(t)=X(t,t_1)x_0+\int_{t_1}^t X(t,s) f(s)ds+\int_{t_1}^t X(t,s) (V\tilde{\varphi})(s)ds.
\end{equation}
%where $\tilde{\varphi}(t)=\left\{\begin{array}{cc}
%0&t\geq t_1\\\varphi(t)&t<t_1.\end{array}\right.$.
\end{theorem}

\begin{proof}
Denote $y(t)=\left\{\begin{array}{cc}
x(t), &t\geq t_1, \\ 0, & t<t_1.\end{array}\right.$
Hence $x(t)=y(t)$, $\dot{x}(t)=\dot{y}(t)$  for  $t\geq t_1$.
Substituting $x(t)=y(t)+\tilde{\varphi}(t)$
into \eqref{2.19}  implies  
\begin{equation}\label{2.20c}
\dot{y}(t)=(Vy)(t)+\tilde{f}(t),\quad t\geq t_1, \quad y(t)=0,  \quad t<t_1, \quad y(t_1)=y_0=x_0,
\end{equation}
where  $\tilde{f}(s)=f(s)+(V\tilde{\varphi})(s)$.
By \cite{AS},  the solution of problem \eqref{2.20c} has the form
\begin{equation}\label{2.20}
y(t)=X(t,t_1)y_0+\int_{t_1}^t X(t,s) \tilde{f}(s)ds.
\end{equation}

Then $y_0=x_0$, $x(t)=y(t)$, $\dot{x}(t)=\dot{y}(t)$ for  $t\geq t_1$ imply \eqref{2.20a}.
\end{proof}

 In Lemma~\ref{lemma2.7a}, we prove that the fundamental function of   \eqref{2.19} is bounded on any
finite triangle  $t_0 \leq s \leq t \leq t_1$.  % of its definition. 

\begin{lemma}\label{lemma2.7a}
For any $t_1> t_0$,   % there exists $M>0$ such 
the  fundamental function of equation \eqref{2.19} is bounded for $t_0\leq s\leq t\leq t_1$:
$\|X(t,s)\|\leq M$ for some $M>0$.
\end{lemma}
\begin{proof}
Denote $x(t)=X(t,s)$, $t\in [s,t_1]$, $t_0\leq s<t_1$.
Then $x(t)$ is a solution of IVP  \eqref{2.19aa}. 
After the substitution 
\begin{equation}\label{2.21a}
x(t)=e^{\lambda(t-s)}y(t), \quad  \lambda>0,
\end{equation}
problem  %\eqref{2.19}-
\eqref{2.19aa}  transforms into
$\displaystyle
e^{\lambda(t-s)}\dot{y}(t)=-\lambda e^{\lambda(t-s)}y(t)+V_s\left(e^{\lambda({\cdot}-s)}y(\cdot)\right)(t)$,  $y(s)=E$.
Hence 
$$
\dot{y}(t)=-\lambda y(t)+e^{-\lambda(t-s)}V_s\left(e^{\lambda (t-s)- \lambda(t-(\cdot))}y(\cdot)\right)(t), \quad y(s)=E.
$$
Since $V_s$ is a linear operator, we have 
%\begin{align*}
%\label{2.22} 
$$
\dot{y}(t)=  -\lambda y(t)+e^{-\lambda(t-s)}e^{\lambda(t-s)}V_s\left(e^{-\lambda(t-(\cdot) ) }y(\cdot)\right)(t) = % \\ = &
 -\lambda y(t)+V_s\left(e^{-\lambda(t-(\cdot) ) } y(\cdot)\right)(t) , \quad y(s)=E,
$$
%$y(s)=E$, 
which,  taking into account  the form of the fundamental matrix of ODE $\dot{y}(t)=  -\lambda y(t)$,
is $\displaystyle 
y(t)=E+\int_s^t e^{-\lambda(t-\xi)}E V_s\left(e^{\lambda(-[t-(\cdot)])}y(\cdot)\right)(\xi)d\xi .
$
As $e^{\lambda(-[t-(\cdot)])}<1$  and $\int_s^t e^{-\lambda(t-\xi)}d\xi  < \frac{1}{\lambda}$,  $t>s$, we get %yields the second inequality in
$$
\|y\|_{[s,t_1]} \leq 1+\|V_s\|_{[s,\infty)} \|y(\cdot)  e^{-\lambda(t-(\cdot))} \|_{[s,t_1]} \int_s^t e^{-\lambda(t-\xi)}d\xi  \leq 1+ \|V_s\|_{[s,\infty)} \|y\|_{[s,t_1]}  \int_s^t e^{-\lambda(t-\xi)}d\xi
$$
and
$\displaystyle 
\|y\|_{[s,t_1]}  \leq 1+\frac{\|V_s\|_{[s,\infty)}}{\lambda}\|y\|_{[s,t_1]}$.
Let $\lambda = \lambda_0$ be such that  $\|V\|_{[t_0,t_1]} < \lambda_0$. Then 
$\displaystyle \frac{\|V_s\|_{[s,\infty)}}{\lambda_0}<1$ and
%\begin{equation}\label{2.23} 
$\displaystyle \|y\|_{[s,t_1]}  \leq  \left(  1-\frac{\|V\|_{[t_0,t_1]}}{\lambda_0} \right)^{-1}, $%\frac{1}{1-\frac{\|V\|_{[t_0,\infty)}}{\lambda_0}}
%\end{equation}
~ where the right-hand side %of \eqref{2.23} 
does not depend on $s\in[t_0,t_1]$. 
Hence
%\begin{equation}\label{2.24} 
$\displaystyle \|y\|_{[t_0,t_1]} \leq M_0 :=   \left(  1-\frac{\|V\|_{[t_0,t_1]}}{\lambda_0} \right)^{-1} $
%. %\frac{1}{1-\frac{\|V\|_{[t_0,t_1]}}{\lambda_0}}.
%\end{equation}
%
and  \eqref{2.21a}  imply
$\displaystyle 
\|X(\cdot,s) \|_{[t_0,t_1]}  \leq M_0  e^{\lambda_0(t_1-t_0)} =M
$,
and the fundamental function $x(t)=X(t,s)$ is %uniformly 
bounded in any finite triangle $t_0\leq s\leq t\leq t_1$.
\end{proof}

\begin{definition}\label{d3.1}
We will say that the fundamental function $X(t,s)$ of equation \eqref{2.19} {\bf exponentially decreases}
(or has a decaying exponential estimate) if there exist $M_0>0$ and   $\lambda_0>0$  such that
\begin{equation}\label{3.1}
\|X(t,s)\|\leq M_0 e^{-\lambda_0 (t-s)},  \quad   t\geq s\geq t_0.
\end{equation}
\end{definition}

\begin{definition}\label{d3.1a}
Assume that $x(t)$ is a solution of IVP \eqref{2.19}-\eqref{2.19a} with $f\equiv 0$.
We will say that equation \eqref{2.19} is  {\bf uniformly exponentially stable (UES)} 
if its solution $x(t)$  satisfies %the estimate 
\begin{equation}\label{3.3a}
\|x(t)\|\leq M e^{-\lambda (t-t_1)} \Big( \|x(t_1)\|
+\|\varphi\|_{(-\infty,t_1)}\Big), \quad t\geq t_1
\end{equation}
%holds 
for any $t_1\geq t_0$, where the constants $M>0$ and $\lambda>0$ do not depend on $t_1\geq t_0$ and the initial function.
%Here $\|\varphi\|_{(-\infty,t_1)}=\sup_{t< t_1} \| \varphi(t)\|$.
If in \eqref{3.3a} only $\lambda$ is independent of $t_1$ and $\varphi$, \eqref{2.19}  is  
{\bf  exponentially stable (ES)}.
\end{definition}

\begin{theorem}\label{th3.02}
\begin{description}
\item[(a)]
If  \eqref{2.19} is UES then its fundamental function %$X(t,s)$ of \eqref{2.19}  
exponentially decreases.

\item{(b)}
If  the fundamental function $X(t,s)$ of \eqref{2.19}  exponentially decreases and the delay $h(t)$ has (A)-property, 
\eqref{2.19} is ES.  If in addition,  the memory is finite $0\leq \tau(t)\leq \tau$ then
 \eqref{2.19} is UES.
\end{description}
\end{theorem}

\begin{proof}
(a)   Set in \eqref{3.3a} $x(t)=X(t,s)$, $\varphi\equiv 0$ and $t_1=s$. Then $x(t)=0$,  $t<s$, while  $\|x(s)\|=\|E\|=1$.
Inequality \eqref{3.3a} implies \eqref{3.1} with $M_0=M, \lambda_0=\lambda$.

b)  Assume that  the fundamental function $X(t,s)$ of   \eqref{2.19}   %the equation  
exponentially decreases and %$0\leq 
$\tau(t)\leq \tau$. 
Then by \eqref{2.20a} with $f(t)\equiv 0$ we get
\begin{align*}
x(t) & =X(t,t_1)x(t_1)+\int_{t_1}^t X(t,s)(V\tilde{\varphi})(s)ds, \quad  t\geq t_1,
\\
%Hence
\|x(t)\| & \leq\|X(t,t_1)\|\|x(t_1)\|+\int_{t_1}^t \|X(t,s)\|\|(V\tilde{\varphi})(s)\|ds, \quad  t\geq t_1.
\end{align*}
Since $\tau(t)\leq \tau$ and $\tilde{\varphi}(s)=0$ for $s>t_1+\tau$, we obtain
\begin{align*}
\|x(t)\|  & \leq M_0e^{-\lambda_0(t-t_1)} \|x(t_1)\|
+M_0\int_{t_1}^{\min\{t,t_1+\tau\}}e^{-\lambda_0(t-s)}\|(V\tilde{\varphi})(s)\|ds
\\
& \leq M_0e^{-\lambda_0(t-t_1)} \|x(t_1)\|+\frac{M_0\|V\|_{[t_0,\infty)}}{\lambda_0} 
\|\varphi\|_{[t_1-\tau,t_1)}\left[ e^{-\lambda_0(t-t_1-\tau)} -    e^{-\lambda_0(t-t_1)}  \right],  %\left|_{t_1}^{t_1+\tau}\right..
\end{align*}

%Hence 
\begin{equation}\label{3.4b}
\|x(t)\|\leq M \left(  \|x(t_1)\|+\|\varphi\|_{[t_1-\tau,t_1)}e^{-\lambda_0(t-t_1)} \right),
\end{equation}
where $ M=\max\left\{ M_0,\frac{M_0}{\lambda_0}e^{\lambda_0 \tau}\|V\|_{[t_0,\infty)}  \right\}$.
The constants M and $\lambda_0$ do not depend on $t_1$, therefore equation \eqref{2.19} is UES. 
In the general case, when the memory $\tau(t)$ depends on the
point $t_1$, the constant $M$ in \eqref{3.4b} also depends on $t_1$, and  \eqref{2.19} is ES.
\end{proof}

\begin{lemma}\label{lemma3.0}
Let the fundamental function $X(t,s)$ of equation \eqref{2.19}  have exponential estimate  \eqref{3.1}
and  let  the right-hand side %of the equation 
$f(t)$ enjoy one of the following properties:

a) $f(t)$ is essentially bounded on $[t_0,\infty)$;

b) $\displaystyle \lim_{t\rightarrow \infty} f(t)=0$;

c)  the right-hand side has an exponential estimate  $\displaystyle  \|f(t)\|\leq M_0 e^{-\lambda_0 t}$,
with some constants $M_0>0,\lambda_0>0$.

Then any solution of  equation \eqref{2.19} has the same property, where in the case c) 
an estimate $\displaystyle \|x(t)\| \leq M_1 e^{-\lambda_1 t}$  holds for some constants $M_1>0$, $\lambda_1>0$.
\end{lemma}

\begin{proof}
Solution representation \eqref{2.20} immediately implies Part a). 

Part b) follows from \cite[Theorem~2.3.5]{G}.

Part c) is verified by direct calculations using \eqref{2.20}, taking $\lambda_1 \leq  \min\{\lambda,\lambda_0\}$ and computing an appropriate $M_1$.
\end{proof}

\begin{remark}\label{remark3.1}
The statement of Lemma~\ref{lemma3.0} is valid if in a)-c) the function $f(t)$ is replaced with the integral  $\int_t^{t+h} f(s)ds$ for some fixed $h>0$.
\end{remark}

The  lemma below is 
a reduced form of   \cite[Theorem~5.3.5]{AMR}.
It gives a modification of the well-known Bohl-Perron theorem applied to equations with abstract Volterra operators.

\begin{lemma}\label{lemma3.1}   \cite[Theorem~5.3.5]{AMR}
Consider FDE  \eqref{2.19}.
Let the linear Volterra bounded operator $V$ have a  $\Delta$-property. 
If, for some $t_1\geq t_0$ and any $f\in {\bf L}_{\infty}[t_1,\infty)$, the solution of  \eqref{2.19} 
with the initial condition $x(t)=0,~t\leq t_1 $
is essentially bounded on $[t_1,\infty)$  then the fundamental function of %equation 
\eqref{2.19}   exponentially decreases.
\end{lemma}

Let us justify that exponential estimate for $t\geq s\geq t_1$ of the fundamental function for some $t_1>t_0$ implies an exponential estimate for $t\geq s\geq t_0$.

\begin{lemma}\label{lemma3.2} 
Assume that the linear Volterra bounded operator $V$ has a $\Delta$-property,
and the fundamental function of equation \eqref{2.19}  has an estimate
$\displaystyle \|X(t,s)\|\leq M_0 e^{-\lambda_0(t-s)}, ~\lambda_0>0$ for $t\geq s\geq t_1$ and some $t_1>t_0$. 
Then there exist $M_1>0$ and  $\lambda_1>0$ such that
%the inequality holds 
$\|X(t,s)\|\leq M_1 e^{-\lambda_1(t-s)},~\lambda_1>0$ for  $t\geq s\geq t_0$.
\end{lemma}

\begin{proof}
For a solution $x$  of %problem 
\eqref{2.19} with the initial condition $x(t)=0$, $t\leq t_0$ and $f\in {\bf L}_{\infty}[t_0,\infty)$, 
we have $x(t)=\int_{t_0}^t X(t,s)f(s)ds$, hence $x(t)=\int_{t_0}^{t_1}X(t,s)f(s)ds+\int_{t_1}^{t}X(t,s)f(s)ds$.
By Lemma \ref{lemma2.7a},  we get  $\|X(t,s)\|\leq M_2$ for  $t_0\leq s\leq t\leq t_1$ for some $M_2>0$ and
\begin{align*}
 \|x(t)\|\leq & \int_{t_0}^{t_1}\|X(t,s)\|\|f(s)\|ds+\int_{t_1}^{t}\|X(t,s)\|\|f(s)\|ds
\\
\leq  & 
M_2 (t_1-t_0)\|f\|_{[t_0,t_1]} +  \|f\|_{[t_0,\infty)} \int_{t_1}^t M_0 e^{-\lambda_0(t-s)}ds 
\\
\leq   &  M_2 (t_1-t_0)\|f\|_{[t_0,\infty)} +\frac{M_0}{\lambda_0}\|f\|_{[t_0,\infty)}  = M\|f\|_{[t_0,\infty)}<\infty, ~M:= M_2 (t_1-t_0)+ \frac{M_0}{\lambda_0}.
\end{align*}
By Lemma \ref{lemma3.1},  the fundamental function of  \eqref{2.19}  
exponentially decreases. % as  $M_1 e^{-\lambda_1(t-s)}$ for $t\geq s\geq t_0$.
\end{proof}

Consider  equation \eqref{2.19} with two terms
\begin{equation}\label{3.2}
\dot{x}(t)=(L_0 x)(t)+(Lx)(t)+f(t), ~ t\geq t_0
\end{equation}
and initial condition \eqref{2.19a} for any $t_1\geq t_0$,
where $L_0, L$ are linear bounded Volterra operators from the space ${\bf C}[t_0,\infty)$ to ${\mathbf L}_{\infty}[t_0,\infty)$ %for any $b >  t_0$, 
with the  memories $\tau_0(t)$, $\tau(t)$ and
delays $h_0(t)=t-\tau_0(t)$, $h(t)=t-\tau(t)$, respectively, $\lim\limits_{t\rightarrow\infty} h_0(t)=\lim\limits_{t\rightarrow\infty} h(t)=\infty$,
 $f$ is a locally essentially bounded function on $[t_0,\infty)$. 
 Conditions on %the functions 
$h_0(t), h(t)$ imply that the delays have (A)-property.
 We also assume that there exist scalar functions  $a_0, a \in {\mathbf  L}_{\infty}[t_0,\infty)$ and a matrix-function $ A \in {\mathbf  L}_{\infty}[t_0,\infty)$ such that
\begin{equation}\label{3.3}
(L_0 E)(t)=A(t), %=\{a_{ij}(t)\}_{i,j=1}^n,
~ \|(L_0 x)(t)\|\leq a_0(t) \sup_{h_0(t)\leq s\leq t} \|x(s)\|,~
\|(L x)(t)\|\leq a(t) \sup_{h(t)\leq s\leq t} \|x(s)\| .
\end{equation}
%where all the parameters above are bounded on $[t_0,\infty)$
%\begin{equation*}\label{3.3a}
% a_0, a \in {\mathbf  L}_{\infty}[t_0,\infty), A %a_{ij}
%\in {\mathbf  L}_{\infty}[t_0,\infty).
% \end{equation*}
 
%We also assume 
Everywhere, we suppose without further mentioning it
that there exists a unique absolutely continuous solution to IVP \eqref{3.2}-\eqref{2.19a}.
For equations with abstract Volterra operators, such sufficient conditions  can be found in \cite{Cord} and for 
%equations with 
specific classes of FDEs  in \cite{AMR}.

The following theorem is the main result of the paper.

\begin{theorem}\label{th3.1} 
Assume that   $L_0$ and $L$ are bounded linear Volterra operators  from ${\bf C}[t_0,\infty)$ to  ${\mathbf L} [t_0,\infty)$ 
 with a $\Delta$-property, the hypotheses in  \eqref{3.3} are satisfied, % and 
% at least one of the following conditions holds
 for some $t_1\geq t_0$ %:
ODE 
\begin{equation}\label{3.4}
\dot{x}(t)=\mu_A(t) x(t)
\end{equation} 
is UES and
\begin{equation}\label{3.5}
\sup_{t\geq t_1}\int_{t_1}^t e^{\int_s^t \mu_A (\xi)d\xi}\left\{ a_0(s)
\int_{h_0(s)}^s \Big[ a_0(\xi)+a(\xi)\Big] \, d\xi+a(s)\right\}ds\leq \alpha<1.
\end{equation}
Then the fundamental function $X(t,s)$ of equation \eqref{3.2} exponentially decreases, and  \eqref{3.2} is ES. If  the memory  of both operators  $L_0$ and $L$  is finite then \eqref{3.2} is UES. 
\end{theorem}

\begin{proof}
By Theorem~\ref{th3.02}, Part (b), we have to verify that  the fundamental function $X(t,s)$ of equation \eqref{3.2}  
exponentially decreases for $t\geq s\geq t_0$.
By Lemma~\ref{lemma3.2},  it is sufficient to prove that the fundamental function $X(t,s)$ of equation \eqref{3.2}  exponentially decreases for $t\geq s\geq t_1$ for some $t_1\geq t_0$. Consider %equation 
\eqref{3.2} with the initial condition $x(t)=0$, $t\leq t_1$
and %with 
$f\in L_{\infty}[t_1,\infty)$. By Lemma~\ref{lemma3.1}, we have to check that the solution $x(t)$ of this problem
is bounded on $[t_1,\infty)$.

Equation \eqref{3.2} with the initial condition $x(t)=0$, %$t_0\leq 
$t\leq t_1$
can be transformed into the form
$$
\dot{x}(t)=A(t)x(t)-\left[ A(t)x(t)-(L_0 x(\cdot) )(t) \right] +(L x(\cdot))(t)+f(t), \quad  t\geq t_1.
$$
We have, for each fixed $t \geq t_1$, from \eqref{3.3},
\begin{align*}
A(t)x(t)-(L_0 x (\cdot))(t) & =(L_0 E)(t)x(t)-(L_0 x(\cdot) )(t)=L_0(x(t)-x(\cdot))(t) 
\\
& =L_0\left(\int_{(\cdot)}^t \dot{x}(\xi)d\xi\right) =L_0\left(\int_{(\cdot)}^t \Big[ (L_0 x)(\xi)+(L x)(\xi)+f(\xi)\Big] d\xi\right)(t).
\end{align*}
Hence equation \eqref{3.2} has the  form
$$
\dot{x}(t)=A(t)x(t)-L_0\left(\int_{(\cdot)}^t \Big[ (L_0 x)(\xi)+(L x)(\xi)+f(\xi) \Big] d\xi\right)(t)
+(L x)(t)+f(t).
$$
Then
\begin{equation}\label{3.9}
x(t)=\int_{t_1}^t X_A (t,s) \left[-L_0\left(\int_{(\cdot)}^s \Big[ (L_0 x)(\xi)+(L x)(\xi)\Big] d\xi\right)(s)
+(L x)(s)\right]ds +\tilde{f}(t),
\end{equation}
where $X_A(t,s)$ is the fundamental function of ODE $\dot{x}(t)=A(t)x(t)$ in \eqref{2.2}.
%\begin{equation}\label{3.4a}
%\dot{x}(t)=A(t)x(t).
%\end{equation}
Equation \eqref{3.4} and thus \eqref{2.2} are UES by Lemma \ref{lemma2.2}, and  
$$\tilde{f}(t)=f_1(t)+f_2(t), ~ f_1(t)=\int_{t_1}^t X_A(t,s) f(s)ds, ~ f_2(t)=
\int_{t_1}^t X_A(t,s) \left[-L_0\left(\int_{(\cdot)}^s f(\xi) d\xi\right)\right](s)ds.$$
Since $f \in {\mathbf L}_{\infty} [t_0,\infty)$, we have
$\displaystyle   \| f_1(t) \| \leq \|f \|_{[t_0,\infty)}  \int_{t_1}^t \| X_A(t,s) \| < \infty$,
where the integral converges by the assumption that \eqref{3.4} is UES and  Lemma~\ref{lemma2.2}, so 
$f_1 \in  {\mathbf L}_{\infty} [t_0,\infty)$. 

Next, proceed to $f_2$.
Let us prove that 
$
L_0\left(\int_{(\cdot)}^t f(\xi) d\xi\right)\in {\mathbf L}_{\infty}[t_0,\infty).
$
By the $\Delta$-condition, there exists $\lambda>0$ such that the inequality $\|y(t)\|\leq e^{-\lambda t}$ implies
$\|(L_0 y)(t)\| \leq M e^{-\lambda t}$ for some $M>0$. 
%For a fixed $t$, 
Denote $\displaystyle y(s) = \int_s^t \frac{f(\xi)}{\|f\|_{[t_0,\infty)}} \lambda e^{-\lambda t} E ~d\xi$, then
$\displaystyle
\| y(s) \| \leq \int_s^t \lambda e^{-\lambda t} ~d\xi  \leq  \int_s^t \lambda e^{-\lambda \xi} ~d\xi   \leq e^{-\lambda s}
$ for  $s \in [t_0,t]$ and any fixed $t\in [t_0,\infty)$ as required.
Hence 
%\begin{align*}
%&
$$
\left\|L_0 \left(\int_{(\cdot)}^t f(\xi) d\xi\right)\right\| = \frac{1}{\lambda} \|f\|_{[t_0,\infty) }  e^{\lambda t} \| L_0 (y)\|
%\\
\leq
%& 
%e^{\lambda t}\|f\|_{[t_0,\infty}\| L_0 ( e^{-\lambda \xi} E)\|\leq
\frac{1}{\lambda} e^{\lambda t}\|f\|_{[t_0,\infty} M e^{-\lambda t}=\frac{M}{\lambda}  \|f\|_{[t_0,\infty)}<\infty.
$$
%\end{align*}
%\begin{align*}
%&
%\left\|L_0 \left(\int_{\cdot}^t f(\xi) d\xi\right)\right\| \leq \| L_0 (e^{\lambda t}\|f\|_{[t_0,\infty)} e^{-\lambda \xi} E)\|
%\\
%= 
%& 
%e^{\lambda t}\|f\|_{[t_0,\infty}\| L_0 ( e^{-\lambda \xi} E)\|\leq
%Me^{\lambda t}\|f\|_{[t_0,\infty}\|e^{-\lambda t}=M \|f\|_{[t_0,\infty}<\infty.
%\end{align*}
Therefore $f_2\in {\mathbf L}_{\infty}[t_0,\infty)$ and $\tilde{f} = f_1+f_2  \in {\mathbf L}_{\infty}[t_0,\infty)$.
%Since  \eqref {3.4a} is UES and $L_0$ is bounded, we get $\tilde{f}=f_1+f_2\in {\mathbf L}_{\infty}[t_0,\infty)$, and
Thus equation \eqref{3.9}  in the space ${\mathbf L}_{\infty}[t_1,\infty)$ has the form 
%\begin{equation}\label{3.10}
$x=Hx+\tilde{f}$,
%\end{equation}
%where we consider equation \eqref{3.10} .
where by
Lemma~\ref{lemma2.2} the operator $H$ satisfies
\begin{equation*} %\label{3.11}
\|(Hx)(t)\|\leq \int_{t_1}^t e^{\int_s^t \mu_A(\xi)d\xi} 
\left[\left\|L_0\left(\int_{(\cdot)}^s \big( (L_0 x)(\xi)+(L x)(\xi) \big) d\xi\right)(s)\right\|
+\|(L x)(s)\|\right]ds +\|\tilde{f}\|_{[t_1,\infty)}.
\end{equation*}

The assumptions in  \eqref{3.3}  imply $\displaystyle \|  (L x)(s) \| \leq a(s)  \|x\|_{[t_0,\infty)}$ and
\begin{align*}%\label{add1}
\left\|L_0\left(\int_{(\cdot)}^s \Big[(L_0 x)(\xi)+(L x)(\xi)\Big] d\xi\right)(s)\right\|
&  \leq a_0(s)\Big \|\int_{h_0(s)}^s \left[(L_0 x)(\xi)+(L x)(\xi)\Big] d\xi\right\|
\\
&  \leq a_0(s)\int_{h_0(s)}^s \Big[a_0(\xi)+  a(\xi)\Big] d\xi ~
\|x\|_{[t_0,\infty)}.
\end{align*}

Inequality \eqref{3.5} yields that the norm of $H$ in %  the  space
${\mathbf L}_{\infty}[t_1,\infty)$ satisfies %is less than one  
$\|H\|_{{\mathbf L}_{\infty}[t_1,\infty)}\leq \alpha<1$.
Thus the equation $x=Hx+\tilde{f}$ has a unique solution in the  space
${\mathbf L}_{\infty}[t_1,\infty)$. 
It means that solutions of equation \eqref{3.9} and therefore of %equation   
\eqref{3.2} are bounded on $[t_1,\infty)$. By Lemma \ref{lemma3.1}, 
 the fundamental function $X(t,s)$ of % equation 
\eqref{3.2} exponentially decreases
for $t\geq s\geq t_1$.  
By Theorem~\ref{th3.02}, Part (b), equation \eqref{3.2} is ES and for the % case of 
bounded memory for both operators $L_0$ and $L$,
the equation is UES.
\end{proof}

\begin{corollary}
\label{cor_th3.1} 
Assume that   $L_0$ and $L$ are bounded linear Volterra operators  from $C[t_0,\infty)$ to  ${\mathbf L}_{\infty} [t_0,\infty)$ 
 with a $\Delta$-property, the hypotheses in  \eqref{3.3} are satisfied and 
% at least one of the following conditions holds
 for some $t_1\geq t_0$ either
\\
(1) equation \eqref{3.4} is UES and 
%\begin{equation}\label{3.6}
$\displaystyle \sup_{t\geq t_1}\left\{ a_0(t) \!\!\! \int\limits_{h_0(t)}^t  \!\! \left[ a_0(\xi)+a(\xi) \right] \, d\xi
+a(t)\right\}
\sup_{t\geq t_1}\int\limits_{t_1}^t  \! e^{\int\limits_s^t \mu_A (\xi)d\xi}ds
\leq \alpha<1$
\\
%\end{equation}
or   
\\
(2) $\mu_A(t)\leq \mu_0<0$ for $t\geq t_1$ and 
one of the inequalities holds:
\begin{equation}\label{3.7}
\sup_{t\geq t_1} \left\{  \frac{a_0(t)}{|\mu_A(t)|}\int_{h_0(t)}^t \left[ a_0(\xi)+a(\xi) \right] \, d\xi
+\frac{a(t)}{|\mu_A(t)|}\right\} \leq \alpha<1,
\end{equation}
 \begin{equation}\label{3.8}
\sup_{t\geq t_1} \left\{  a_0(t)\int_{h_0(t)}^t \left[ a_0(\xi)+a(\xi) \right] \, d\xi
+a(t)\right\}  <|\mu_0| .
\end{equation}
%holds.
Then the fundamental function $X(t,s)$ of equation \eqref{3.2} exponentially decreases and  \eqref{3.2} is ES. If  the memory  of both operators  $L_0$ and $L$  is finite then \eqref{3.2} is UES. 
\end{corollary}

\begin{proof}
(1) immediately follows from  Theorem~\ref{th3.1}. 

(2) To justify the result, we notice that, once \eqref{3.7} holds,
\begin{align*}
& \sup_{t\geq t_1}\int_{t_1}^t e^{\int_s^t \mu_A (\xi)d\xi}\left\{  a_0(s)\int_{h_0(s)}^s [a_0(\xi)+a(\xi)] \, d\xi
+a(s)\right\}ds
\\
=  & \sup_{t\geq t_1}\int_{t_1}^t  e^{\int_s^t \mu_A (\xi)d\xi}|\mu_A(s)|\left\{  \frac{a_0(s)}{ |\mu_A(s)|} \int_{h_0(s)}^s [a_0(\xi)+a(\xi)] \, d\xi
+\frac{a(s)}{|\mu_A(s)|} \right\} ~ds\\
\leq  & \sup_{t\geq t_1} \left\{ \frac{a_0(t)}{|\mu_A(t)|}\int_{h_0(t)}^t [a_0(\xi)+a(\xi)]\, d\xi
+\frac{a(t)}{|\mu_A(t)|}\right\} \leq \alpha<1,
\end{align*}
which concludes the proof in the case of \eqref{3.7}. Since  \eqref{3.8} implies  \eqref{3.7}, the proof is complete.
\end{proof}

\section{Developments of the Main Result for Equations with Abstract Volterra Operators}
\label{sec:4}

In this section, we also assume that all linear  Volterra operators are bounded and act from ${\bf C}[t_0,\infty)$
to $\bf{L}_{\infty}[t_0,\infty)$. 

Consider first a vector equation including one linear  Volterra operator $L_0$ only with memory $\tau_0(t)\geq 0$, 
$h_0(t)=t-\tau_0(t),~\lim\limits_{t\rightarrow\infty}h_0(t)=\infty$ and $f\in {\mathbf L}_{\infty}[t_0,\infty)$  in the right-hand side
\begin{equation}\label{4.1}
\dot{x}(t)=(L_0 x)(t)+ f(t),~t\geq t_0.
\end{equation}

\begin{theorem}
\label{c4.1}
Assume that  $L_0$ is a linear bounded Volterra operator  with delay $h_0(t)$ and a $\Delta$-property, %the conditions
$(L_0 E)(t)=A(t)$ and  $\|(L_0 x)(t)\|\leq a_0(t) \sup_{h_0(t)\leq s\leq t} \|x(s)\|$  for some locally integrable $a_0(t)\geq 0$  and 
%at least one of the following conditions holds 
for some $t_1\geq t_0$, ODE  \eqref{3.4}
%\begin{equation}\label{4.2}
%\dot{x}(t)+\mu_{A}(t) x(t)=0
%\end{equation} 
is UES and
\begin{equation} 
\label{4.3}
\sup_{t\geq t_1}\int_{t_1}^t e^{\int_s^t \mu_{A} (\xi)d\xi} a_0(s)\int\limits_{h_0(s)}^s a_0(\zeta)d\zeta ~ ds\leq \alpha<1.
\end{equation}
Then the fundamental function $X(t,s)$ of equation \eqref{4.1} exponentially decreases and  \eqref{4.1} is ES. If  the memory  of the operator  $L_0$   is finite then \eqref{4.1} is UES. 
\end{theorem}

\begin{corollary}
\label{c4.1a}
Let in the statement of Theorem~\ref{c4.1} all the conditions hold, while the assumption that \eqref{3.4} 
is UES and \eqref{4.3} holds is replaced with anyone of the two following conditions:

a) equation \eqref{3.4} is UES and 
%\begin{equation}\label{4.4}
$\displaystyle \sup_{t\geq t_1}\int_{t_1}^t e^{\int_s^t \mu_{A} (\xi)d\xi}ds ~
\sup_{t\geq t_1} \left\{ a_0(t)\int\limits_{h_0(t)}^t a_0(\zeta)d\zeta \right\} \leq \alpha<1$;
%\end{equation}

b) $\mu_A(t)\leq \mu_0<0$ for  $t\geq t_1$ and anyone of the inequalities below is satisfied:
$$ \sup_{t\geq t_1} \left\{  \frac{a_0(t)}{|\mu_A(t)|}\int_{h_0(t)}^t a_0(\xi)d\xi \right\} \leq \alpha<1  \mbox{\rm  ~  or ~ }
\sup_{t\geq t_1}  \left\{   a_0(t)\int_{h_0(t)}^t a_0(\xi)d\xi  \right\} < |\mu_0|.
$$
Then the conclusions of Theorem~\ref{c4.1} are valid.
\end{corollary}

The following result can be applied to obtain conditions of stability  preservation  under perturbations for different classes of FDEs.

\begin{theorem}\label{th4.1}
Assume that for equation \eqref{3.2} the hypotheses in \eqref{3.3}  are satisfied, together with  conditions of either Theorem~\ref{c4.1} or Corollary \ref{c4.1a}. 
Let also  the memory $\tau_0(t)$ of the operator $L_0$ be finite: $\tau_0(t)\leq\tau$ and let at least one of the following conditions hold:
 \\
(1) $\displaystyle \lim_{t\rightarrow \infty} a(t)=0$;\vspace{1mm} \\
(2) $\displaystyle  \int_{t_0}^{\infty} a(s)ds<\infty$;\\
(3) for some $T>0$, $\displaystyle  \lim_{t\rightarrow \infty}\int_t^{t+T} a(s)ds =0$. \\
 
 Then the fundamental function $X(t,s)$ of equation \eqref{3.2} exponentially decreases and  \eqref{3.2} is ES. If  the memory  of % both operators  $L_0$ and 
the operator $L$  is also finite then \eqref{3.2} is UES. 
\end{theorem}

\begin{proof}
It is sufficient to prove Part (3), since Parts  (1) and (2) are corollaries of (3), as they include more restrictive conditions on % the bound of 
$a(t)$ in \eqref{3.3}.
Assume that the conditions of Theorem~\ref{c4.1} for some $\alpha>0$  and Part (3) in the statement of the theorem hold (the assumptions in Corollary~\ref{c4.1a} 
imply  \eqref{4.3} and thus are more restrictive  than Theorem~\ref{c4.1}). 
Denote for some $t^*\geq t_1$
$$
\beta(t):= \int_{t^*}^t e^{\int_s^t \mu_A (\xi)d\xi}\left\{ a_0(s)
\int_{h_0(s)}^s \left[ a_0(\zeta)+a(\zeta)\right] \, d\zeta +a(s)\right\}ds.
$$
By \eqref{3.5} in Theorem~\ref{th3.1}, we have to find $\beta$ and  $t^*\geq t_1$ such that $\beta(t)\leq \beta<1$ for  $t\geq t^*$. 
Denote 
$\beta(t) = \beta_1(t)+\beta_2(t)+\beta_3(t)$, where
%\begin{align*}
$\displaystyle \beta_1(t)   =\int_{t^*}^t e^{\int_s^t \mu_A (\xi)d\xi} a_0(s)
\int_{h_0(s)}^s a_0(\zeta)\,d\zeta~ds$,
\\
$\displaystyle \beta_2(t)   =\int_{t^*}^t e^{\int_s^t \mu_A (\xi)d\xi} a_0(s)
\int_{h_0(s)}^s a(\zeta)\, d\zeta~ds$, 
%\\
$\displaystyle \beta_3(t)   =\int_{t^*}^t e^{\int_s^t \mu_A (\xi)d\xi}a(s)ds$.
%\end{align*}
By \eqref{4.3}, %the assumptions of the theorem 
$\beta_1(t)\leq \alpha<1$ for  $t\geq t_1$, hence also  
$\beta_1(t)\leq \alpha<1$ for $ t\geq t^*\geq t_1$.
The inequality $t-h_0(t)\leq \tau$ implies, where $\lfloor t \rfloor$ is the largest integer not exceeding $t$,  
$ \displaystyle 
~\int_{h_0(s)}^s a(\xi)d\xi $  $\displaystyle \leq \left\lfloor \frac{\tau}{T}+1\right\rfloor \sup_{h_0(s)\leq t\leq s}\int_{t}^{t+T} a(\xi)d\xi,
$
and  $\displaystyle \lim_{t\rightarrow\infty}\int_t^{t+T} a(\xi)d\xi=0$ ~leads to
$\displaystyle \lim_{s \rightarrow\infty}\int_{h_0(s)}^s a(\xi)d\xi=0$. 
The fact that ODE $\dot{x}(t)+\mu_A(s)x(s)=0$ is UES implies  $e^{\int_s^t \mu_A (\xi)d\xi}\leq Me^{-\lambda(t-s)}$
for some $M>0,\lambda>0$. By Lemma~\ref{lemma2.1a} and boundedness of $a_0(t)$, $\lim_{t\rightarrow\infty}\beta_2(t)=0$.
Again by Lemma~\ref{lemma2.1a},  $\lim_{t\rightarrow\infty}\beta_3(t)=0$.
Therefore, for some $t^*\geq t_1$ and sufficiently small $\varepsilon>0$, $\varepsilon < 1-\alpha$ we can find $t^{\ast}$ such that 
$\beta_2(t) <\frac{\varepsilon}{2}$ and $\beta_3(t) <\frac{\varepsilon}{2}$ for $t \geq t^{\ast}$. Therefore
$\beta(t)=\beta_1(t)+\beta_2(t)+\beta_3(t)\leq \alpha+\varepsilon<1$, which concludes the proof.
\end{proof}

Consider now %for any $t_1\geq t_0$ 
an equation with a non-delay term and one linear bounded Volterra operator
\begin{equation}\label{4.7}
\dot{x}(t)=A(t)x(t)+(L x)(t)+f(t),~ t\geq t_0,
\end{equation}
%with initial condition \eqref{2.19a},
where $A(t)$ is an essentially bounded matrix-function on $[t_0,\infty)$,
$L:{\bf C}[t_0,\infty) \to {\mathbf L}_{\infty}[t_0,\infty)$ is a linear bounded Volterra operator %from the space
 with the  memory $\tau(t)$ and the delay
$ h(t)=t-\tau(t)$,  where $\displaystyle \lim_{t\rightarrow\infty} h(t)=\infty$,
 $f$ is an essentially bounded function on $[t_0,\infty)$. 
 Conditions on functions $h(t)$ imply that the delay has (A)-property.
 We also assume that  there exists $ a \in {\mathbf  L}_{\infty}[t_0,\infty)$  such that
\begin{equation}\label{4.7a}
\|(L x)(t)\|\leq a(t) \sup_{h(t)\leq s\leq t} \|x(s)\| .
\end{equation}

\begin{theorem}\label{th4.2}
Assume  that $L:{\bf C}[t_0,\infty)  \to {\mathbf L}_{\infty}[t_0,\infty)$ is a linear bounded Volterra operator  with the  memory $\tau(t)$ and  a $\Delta$-property, condition \eqref{4.7a} is satisfied for some $a \in {\mathbf L}_{\infty}[t_0,\infty)$,
%$\|(L x)(t)\|\leq a(t) \sup_{h(t)\leq s\leq t} \|x(s)\|$ 
%holds for some essentially bounded on $[t_0,\infty)$ 
%scalar function 
$a(t)\geq 0$, 
and that there exists $t_1\geq t_0$ such that at least one of the following conditions holds:
\\
a) ODE 
%\begin{equation}\label{4.8}
$\dot{x}(t)+\mu_A(t) x(t)=0$
%\end{equation} 
is UES and either
%\begin{equation*} %\label{4.9}
$\displaystyle \sup_{t\geq t_1}\int_{t_1}^t e^{\int_s^t \mu_A (\xi)d\xi} ~ a(s)ds\leq \alpha<1$
%\end{equation*}
 or \\
%b) equation \eqref{4.8} is UES and 
%\begin{equation*} %\label{4.10}
$ \displaystyle  \| a(t) \|_{[t_1,\infty)} \sup_{t\geq t_1}\int_{t_1}^t e^{\int_s^t \mu_A (\xi)d\xi} \, ds
%\sup_{t\geq t_1} a(t)
\leq \alpha<1$;
%\end{equation*}
\\
b) $\mu_A(t)\leq \mu_0<0$ and  either
%\begin{equation}\label{4.11}
~ $\displaystyle  \sup_{t\geq t_1} \frac{a(t)}{|\mu_A(t)|}\leq \alpha<1$ ~~
%\end{equation}
or
%d) $\mu_A(t)\leq \mu_0<0$ and 
%\begin{equation}\label{4.12}
~ $\displaystyle \sup_{t\geq t_1}  a(t)<|\mu_0|$.
%\end{equation}

Then the fundamental function of equation \eqref{4.7} exponentially decreases and  \eqref{4.7} is ES. If  the memory  of the operator   $L$  is finite then \eqref{4.7} is UES. 
\end{theorem}

Consider now % for any $t_1\geq t_0$ 
an equation including several terms % linear bounded Volterra operators
\begin{equation}\label{4.13}
\dot{x}(t)=\sum_{k=1}^m (L_k x)(t)+f(t), \quad t\geq t_0,
\end{equation}
where   each $L_k:{\bf C}[t_0,\infty)  \to {\mathbf L}_{\infty}[t_0,\infty)$, $k=1\dots,m$ is a bounded linear Volterra operators  with $\tau_k(t)$ memory and  a $\Delta$-property, 
% and the conditions
$\|(L_k x)(t)\|\leq a_k(t) \sup_{h_k(t)\leq s\leq t} \|x(s)\|$ %hold 
for some $a_k \in {\mathbf L}_{\infty}[t_0,\infty)$, %essentially bounded on $[t_0,\infty)$ functions 
$a_k:[t_0,\infty) \to [0,\infty)$, $\lim\limits_{t\rightarrow\infty} h_k(t)  %=t-\tau_k(t))
=\infty$, $k = 1, \dots m$. 
Next, introduce a non-empty set $\Omega\subset \{1,\dots,m\}$, let   ${\Omega}'= \{1,\dots,m\} \backslash \Omega$,
$\displaystyle L_0=\sum_{k\in \Omega} L_k$,  $\displaystyle  L=\sum_{k\in {\Omega}'} L_k$,
$A(t), a_0(t), a(t)$ are %denoted 
as in \eqref{3.3}.
The operators
$L_0$ and $L$  have delays $h_0(t)=\min_{k\in \Omega} h_k(t)$ and %,  the operator
%$L$ has delay 
$h(t)=\min_{k\in {\Omega}'} h_k(t)$, respectively. Both operators $L_0$ and $L $ have  a $\Delta$-property.
As a direct corollary of Theorem~\ref{th3.1} and Corollary~\ref{cor_th3.1}, we get %have the following

\begin{theorem}\label{th4.3}
Let  there exist a non-empty set $\Omega\subset \{1,\dots,m\}$ such that the conditions
 of either Theorem \ref{th3.1} or Corollary~\ref{cor_th3.1} hold. 
Then the fundamental function  of equation \eqref{4.13} exponentially decreases and  \eqref{4.13} is ES. If  the memory  of  each  $L_k$, $k=1,\dots,m$  is finite then \eqref{4.13} is UES. 
\end{theorem}

\begin{remark}\label{r4.1}
The set $\{1,\dots,m\}$ has $2^m-1$ non-empty subsets. Hence Theorem \ref{th4.3} gives $2^m-1$ independent
exponential stability tests.
\end{remark}

The next case in this section is a system of scalar equations with Volterra operators
\begin{equation}\label{4.14}
\dot{x}_i(t)=\sum_{j=1}^n(L_{ij} x_j)(t)  + f_i(t),~i=1,\dots,n,  \quad  t\geq t_0, %t_1\geq t_0,
\end{equation}
%with initial condition \eqref{2.19a},
where we assume $(L_{ij} x_j)(t)\equiv 0$ in the absence of the corresponding term in the system.
Here  $L_{ij}$ are bounded scalar  Volterra operators from 
${\mathbf C}[t_0,\infty)$ to ${\mathbf  L}_{\infty}[t_0,\infty)$
 with $\tau_{ij}(t)$ memory and  a  $\Delta$-property, 
 $h_{ij}(t)=t-\tau_{ij}(t)$,  $\displaystyle \lim_{t\rightarrow\infty}h_{ij}(t)=\infty$, $i,j=1,\dots,n$. 
Denote  
$x(t)=( x_1(t),\dots,x_n(t) )^T$,  % $T$ is transposition, 
$ f(t)= (f_1(t),\dots,f_n(t))^T$ and 
$\displaystyle 
 (L_0x)(t)=\left( \sum_{j=1}^n (L_{1j}x_j)(t),\dots,\sum_{j=1}^n (L_{nj}x_j)(t) \right)^T
 =\left(\sum_{j=1}^n (L_{ij}x_j)(t)\right)_{i=1}^n. %, i=1,\dots,n.
 $
 Then %the operator   
$L_{0}$ is a bounded  Volterra operator from  ${\bf C}[t_0,\infty)$ to ${\bf L}_{\infty}[t_0,\infty)$
 with memory $\tau_0(t) \leq \sup_{ij}\tau_{ij}(t)$, and system \eqref{4.14}
 is equivalent to the vector equation
 \begin{equation}\label{4.14a}
 \dot{x}(t)=(L_0x)(t)+f(t), ~t\geq t_0.
 \end{equation}

 \begin{theorem}\label{th4.4}
Assume that at least one of the conditions of either Theorem~\ref{c4.1} or Corollary \ref{c4.1a} hold, where
$\displaystyle
A(t)=\left(\begin{array}{ccc}
(L_{11}1)(t)&\dots&(L_{1n}1)(t)\\
\dots&\dots&\dots\\
(L_{n1}1)(t)&\dots&(L_{nn}1)(t)\end{array}\right)$, $1$ is a scalar function identically equal to one, $a_0(t)=\|A(t)\|.
$
Then the fundamental function of equation \eqref{4.14a} exponentially decreases and  system \eqref{4.14} is ES. 
If  the memory  of  each  $L_{ij}$, $i,j=1, \dots, n$  is finite then \eqref{4.14} is UES. 
\end{theorem}

 \begin{proof}
The proof  follows from Theorem~\ref{c4.1}. We will only find $L_0 E$. 
If $B(t)=\{ b_{ij} (t)  \}_{i,j=1}^n$ is a matrix function, then the calculation of the  matrix function $(L_0 B)(t)$ 
is similar to the calculation of the product of two matrices.
It means that $\displaystyle (L_0 B)(t)=\left\{ \sum_{k=1}^n (L_{ik}b_{kj})(t)\right\}_{i,j=1}^n$. 
Hence $\displaystyle  (L_0 E)(t)=\left\{ \sum_{k=1}^n (L_{ik}\delta_{kj})(t)\right\}_{i,j=1}^n$, where 
$\displaystyle \delta_{kj}=\left\{\begin{array}{cc}1,&k=j,\\0,&k\neq j.\end{array}\right.$
Therefore $A(t)=(L_0 E)(t)=\left\{  (L_{ij}1)(t)\right\}_{i,j=1}^n$ as defined above.
\end{proof}

Denote $\Omega\subset {\Omega}^n$, where ${\Omega}^n=\{(i,j)\}_{i,j=1}^n$ 
is the set of all pairs of indices $(i,j)$, ${\Omega}'={\Omega}^n \backslash \Omega$,
$\displaystyle 
(L_0x)(t)=\left( \sum_{(1,j)\in \Omega} (L_{1j}x_j)(t),\dots,\sum_{(n,j)\in \Omega} (L_{nj}x_j)(t) \right)^T,
$ 
where $\displaystyle \sum_{(i,j)\in \Omega} (L_{ij}x_j)(t)=0$ if there are no $(i,j)\in \Omega$, %set of  pairs of indices is empty,
with memory $\displaystyle \tau_0(t)=\sup_{(ij)\in\Omega}\tau_{ij}(t)$,
$\displaystyle 
(Lx)(t)=\left( \sum_{(1,j)\in \Omega'} (L_{1j}x_j)(t),\dots,\sum_{(n,j)\in \Omega'} (L_{nj}x_j)(t) \right)^T
$\!, where $\displaystyle \sum_{(i,j)\in \Omega'} (L_{ij}x_j)(t)=0$ if there are no $(i,j)\in \Omega'$,
%$\displaystyle \sum_{(1,j)\in \Omega'} (L_{1j}x_j)(t)=0$ if the set of pairs of indexes is empty,
with memory $\displaystyle \tau(t)=\sup_{(ij)\in\Omega'}\tau_{ij}(t)$.
Then system \eqref{4.14} is equivalent to the equation
% \begin{equation}\label{4.14b}
$\displaystyle  \dot{x}(t)=(L_0x)(t)+(Lx)(t)+f(t)$ for $t\geq t_0$ as in \eqref{3.2}.
% \end{equation}
Introduce the matrices 
$
A(t)=(L_0 E)(t)=\left\{  (L_{ij}1)(t)\right\}_{i,j\in\Omega},~
A_1(t)=\left\{  |(L_{ij}1)(t)|\right\}_{i,j\in\Omega'},
$
then $\|(Lx){t}\|\leq a(t)\sup_{h(t)\leq s\leq t}\|x(t)\|$, where  $ a(t)=\|A_1(t)\|$. %, $h(t)=t-\tau(t)$. 
Theorem~\ref{th3.1} implies

 \begin{theorem}\label{th4.5}
Let  $\Omega\subset \{(i,j)\}_{i,j=1}^m$  be a non-empty set  such that the conditions
 of Theorem \ref{th3.1} or Corollary~\ref{cor_th3.1} hold. 
Then the fundamental function  of system  \eqref{4.14a} exponentially decreases and  \eqref{4.14} is ES. 
If  the memory  of each  $L_{ij}$, $i,j=1, \dots, n$    is finite then \eqref{4.14} is UES. 
\end{theorem}

%The proof follows from Theorem~\ref{th3.1}.

\begin{remark}\label{remark5.1}
Since any set $\Omega$ in Theorem~\ref{th4.5} must include all the diagonal entries,
otherwise \eqref{2.2} cannot be UES,
% $(ii)$, 
Theorem~\ref{th4.5}  contains  $2^{n^2-n}$  tests when the fundamental function exponentially decreases. 
%since  $\Omega$ should include all the diagonal entries, otherwise a negative matrix measure is impossible.
\end{remark}

 \section{Applications to different classes of FDEs}
\label{sec:5}

Most definitions  for differential equations with abstract Volterra operators
are also used  for specific classes of equations. 
In particular, we apply the definitions of UES, ES and of an exponentially decreasing
fundamental function for all the equations considered in this chapter.
FDEs, which we consider here can be divided into the following groups: 
(1)~equations with concentrated delays; (2)~equations with distributed delays; (3)~mixed equations, where concentrated and distributed delays are combined. %different kinds of delays are included.

\subsection{Equations with concentrated delays}
\label{sec:5.1}

Consider %for any $t_1\geq t_0$ 
a  linear vector equation
\begin{equation}\label{5.1}
\dot{x}(t)=\sum_{k=1}^m A_k(t)x(h_k(t)) + f(t),~t\geq t_0
\end{equation}
and for any $t_1\geq t_0$ an initial condition 
\begin{equation}\label{5.2}
 x(t)=\varphi(t), \quad  t< t_1, \quad  x(t_1)=x_0,
\end{equation}
where for  $h_k(t)$, $A_k(t)$,  and  $\varphi(t)$ conditions
$(ii), (iii), (v)$ hold, and %the right-hand side 
$f(t)\in {\mathbf L}_{\infty}[t_0,\infty)$.
% is an essentially bounded on $[t_0,\infty)$
%vector function. 
The same conditions will be assumed for all equations  considered in Section \ref{sec:5.1}. 

As usual  \cite{AS} (see also the definition for abstract equations),  the fundamental function $X(t,s)$ of equation \eqref{5.1} is 
a solution for $t\geq s\geq t_0$ of the matrix equation
$$
\frac{\partial X}{\partial t} (t,s)=\sum_{k=1}^m A_k(t)X(h_k(t),s)=0, \quad t \geq s,\quad X(s,s)=E, \quad X(h_k(t),s)=0,~~ h_k(t)<s.
$$

\begin{lemma}  \cite[p.  27]{AS}, 
\label{lemma5.1}
There exists a unique solution of IVP \eqref{5.1}-\eqref{5.2}, and for this solution the following representation 
holds, where $\varphi(h_k(s))=0$ if $h_k(s)\geq t_1$:
\begin{equation}\label{5.3}
\displaystyle x(t)=X(t,t_1)x(t_1)+\int_{t_1}^t X(t,s)f(s)ds+\sum_{k=1}^m \int_{t_1}^t X(t,s)A_k(s)\varphi(h_k(s))ds,  \quad t\geq t_1.
\end{equation}
\end{lemma}
\begin{remark}\label{r5.1a}
Formula \eqref{5.3} also follows from \eqref{2.20}.
\end{remark}

Define linear bounded  Volterra operators  $L_k: {\bf C}[t_0,\infty) \to  {\bf L}_{\infty}[t_0,\infty)$ as
 \begin{equation}\label{5.6}
(L_k y)(t)=\left\{\begin{array}{ll}
A_k(t)y(h_k(t)),& t\geq t_0, \\ 0,& h_k(t)<t_0.\end{array}\right.
\end{equation}
We assume in this section that the conditions of 
Lemma \ref{lemma2.4} hold, i.e. all the operators $L_k$ in \eqref{5.6}  have a $\Delta$-property (see Definition~\ref{d2.4}).
Condition $(ii)$ implies (A)-condition for all the delays $h_k(t)$. 
%Hence for all equations considered in this subsection, (A)-condition  and $\Delta$-property hold.
Problem \eqref{5.1}--\eqref{5.2} can be transformed into IVP
%the initial value problem 
for an %differential 
equation with abstract Volterra operators
\begin{equation}\label{5.7}
\dot{y}(t)=\sum_{k=1}^m (L_k y)(t)+f(t), \quad t\geq  %t_1\geq 
t_0
\end{equation}
with initial condition \eqref{5.2}.
As $(L_k E)(t)=A_k(t)$, we get $\displaystyle \| (L_k y)(t)\|\leq \| A_k(t)\|\sup_{h_k(t)\leq s\leq t} \|y(s)\|$.
Since equation \eqref{5.1} is equivalent to \eqref{5.7} with abstract Volterra operators, 
all the statements of  the previous sections for equations with abstract Volterra  operators are valid
for  \eqref{5.1}. %,  where the functions $A_k(t)$ are replaced with $B_k(t)$. Other notations will be given as needed.
As corollaries of Theorems \ref{th4.1}-\ref{th4.4}  for  equations with abstract Volterra operators, 
we get stability tests for \eqref{5.1}.

Introduce non-empty sets  $\Omega\subset \{1,\dots,m\}$,  ${\Omega}'= \{1,\dots,m\} \backslash \Omega$ and functions
%Denote
 $\displaystyle 
 A(t)=\sum_{k\in \Omega}A_k(t)$,  $\displaystyle a_0(t)=\sum_{k\in \Omega}\|A_k(t)\|$, 
 $\displaystyle  a(t)=\sum_{k\in {\Omega}'}\|A_k(t)\|$,
 $\displaystyle 
h_0(t)=\min_{k\in \Omega} h_k(t)$,   $\displaystyle h(t)=\min_{k\in {\Omega}'} h_k(t)$.
Theorem \ref{th4.3} implies

\begin{proposition}
\label{th5.1} 
Assume that  at least one of the following conditions holds for some $t_1\geq t_0$:
\\
a) %ODE 
%\begin{equation}\label{5.7a}
$\dot{x}(t)=\mu_A(t) x(t)$
%\end{equation} 
is UES and  %either 
%\begin{equation}\label{5.7b}
$\displaystyle \sup_{t\geq t_1}\int\limits_{t_1}^t e^{\int_s^t \mu_A (\xi)d\xi}\left\{ a_0(s) \!\!\!\!
\int\limits_{h_0(s)}^s \!\!\!\!  \left[ a_0(\xi)+a(\xi)\right] \, d\xi+a(s)\right\}ds\leq \alpha<1$
%\end{equation}
or
\\
%b) equation \eqref{5.7a} is UES and 
\begin{equation}\label{5.7c}
\sup_{t\geq t_1}\left\{ a_0(t)\int_{h_0(t)}^t  \left[ a_0(\xi)+a(\xi) \right] \, d\xi
+a(t)\right\}~
\sup_{t\geq t_1}\int_{t_1}^t e^{\int_s^t \mu_A (\xi)d\xi}ds
\leq \alpha<1;
\end{equation}
b) $\mu_A(t)\leq \mu_0<0$ for $t\geq t_1$ and either 
%\begin{equation*}\label{5.7d}
$\displaystyle \sup_{t\geq t_1} \left\{  \frac{a_0(t)}{|\mu_A(t)|}\int_{h_0(t)}^t \left[ a_0(\xi)+a(\xi) \right] \, d\xi
+\frac{a(t)}{|\mu_A(t)|}\right\} \leq \alpha<1 $
%\end{equation*}
or
%d) $\mu_B(t)\leq \mu_0<0, t\geq t_1$ and 
%\begin{equation*}\label{5.7e}
$\displaystyle \sup_{t\geq t_1} \left\{  a_0(t)\int_{h_0(t)}^t \left[ a_0(\xi)+a(\xi) \right] \, d\xi
+a(t)\right\}  <|\mu_0| $.
%\end{equation*}

Then the fundamental function  of equation \eqref{5.1} exponentially decreases and \eqref{5.1} is ES. If there exists $\tau>0$ such that $t-h_k(t)\leq \tau$ for some $\tau>0$,
$k=1, \dots,m$
then \eqref{5.1}  is UES.
\end{proposition}

Next, denote $\displaystyle A(t)=\sum_{k=1}^m A_k(t)$, $\displaystyle a_0(t)=\sum_{k=1}^m \|A_k(t)\|$, $h_0(t)=\min_k h_k(t)$.
As a corollary of  Proposition~\ref{th5.1} for $\Omega=\{1,\dots,m\}$, we get

\begin{corollary}\label{c5.1}
Assume that  at least one of the following conditions holds for some  $t_1\geq t_0$:
\\
a) ODE  $\dot{x}(t)=\mu_A(t) x(t)$ %\eqref{5.7a} 
is UES and either
%\begin{equation} 
%\label{5.8a}
$\displaystyle \sup_{t\geq t_1}\int\limits_{t_1}^t e^{\int_s^t \mu_{A} (\xi)d\xi} a_0(s)\int\limits_{h_0(s)}^s  a_0(\zeta)d\zeta ~ ds\leq \alpha<1$
or
\\
%\end{equation}
%b) equation \eqref{5.7a} is UES and 
\begin{equation}\label{5.6b}
\sup_{t\geq t_1}\int_{t_1}^t e^{\int_s^t \mu_{A} (\xi)d\xi}ds ~
\sup_{t\geq t_1}\left\{ a_0(t)\int\limits_{h_0(t)}^t a_0(\zeta)d\zeta\right\}\leq \alpha<1;
\end{equation}
b) $\mu_A(t)\leq \mu_0<0, t\geq t_1$ and  either 
%\begin{equation*} %\label{4.5}
$\displaystyle  \sup\limits_{t\geq t_1} \frac{a_0(t)}{|\mu_A(t)|}\int\limits_{h_0(t)}^t a_0(\xi)d\xi\leq \alpha<1$
%\end{equation*}
or
%d) $\mu_B(t)\leq \mu_0<0, t\geq t_1$ and 
%\begin{equation}\label{5.8c}
$\displaystyle   \sup_{t\geq t_1}  a_0(t) \!\! \int\limits_{h_0(t)}^t \!\! a_0(\xi)d\xi< |\mu_0| $.
%\end{equation}

Then the fundamental function of equation \eqref{5.1}  exponentially decreases  and   \eqref{5.1}  is ES.
If, in addition,  $0\leq t-h_k(t)\leq \tau$, $k=1, \dots, m$ for some $\tau>0$,  \eqref{5.1}  is UES. 
\end{corollary}

%It is well known  that $\int_0^{\infty}\sin (t^2) ds=\sqrt{\frac{\pi}{8}}<\infty$. 

Consider %for any $t_1\geq t_0$ 
a vector  equation with a non-delay term
\begin{equation}\label{5.10a}
\dot{x}(t)=A(t)x(t)+\sum_{k=1}^m A_k(t)x(h_k(t))+f(t), ~t\geq t_0
\end{equation}
and denote $a(t)=\sum_{k=1}^m \|A_k(t)\|$.
Theorem~\ref{th4.2} implies

\begin{proposition}
\label{th5.2} 
Assume that  there exist $t_1\geq t_0$ and $\alpha \in (0,1)$  such that at least one of the following conditions holds for some $t_1\geq t_0$:
\\
a) ODE 
%\begin{equation}\label{5.11a}
$\dot{x}(t)+\mu_A(t) x(t)=0$
%\end{equation} 
is UES and either
%\begin{equation*} %\label{4.9}
$\displaystyle \sup_{t\geq t_1}\int_{t_1}^t e^{\int_s^t \mu_A (\xi)d\xi} a(s)ds\leq \alpha$
%\end{equation*}
%
%b) equation \eqref{5.11a} is UES and 
%\begin{equation*} %\label{4.10}
 or \\  $\displaystyle \sup_{t\geq t_1}\int_{t_1}^t e^{\int_s^t \mu_A (\xi)d\xi} ds
\sup_{t\geq t_1} a(t)\leq \alpha$;
%\end{equation*}
\\
b) $\mu_A(t)\leq \mu_0<0$ and either  ~
%\begin{equation}\label{5.11b}
$\displaystyle \sup_{t\geq t_1} \frac{a(t)}{|\mu_A(t)|}\leq \alpha$
%\end{equation}
~ or ~
%d) $\mu_B(t)\leq \mu_0<0$ and 
%\begin{equation}\label{5.11c}
$\displaystyle \sup_{t\geq t_1}  a(t)<|\mu_0|$.
%\end{equation}

Then the fundamental function  of equation \eqref{5.10a} exponentially decreases and  \eqref{5.10a} is ES. If $t-h_k(t)\leq \tau$, $k=1, \dots,m$ for some $\tau>0$,
\eqref{5.10a}  is UES. 
\end{proposition}

Consider now %for any $t_1\geq t_0$ 
a system of scalar linear DDEs
\begin{equation}\label{5.10}
\dot{x_i}(t)=\sum_{j=1}^n a_{ij}(t)x_j(h_{ij}(t))+f_i(t), \quad  t\geq t_0, \quad  i=1,\dots,n
\end{equation}
%and for any $t_1\geq t_0$ 
with the initial conditions for any $t_1\geq t_0$
\begin{equation}\label{5.11}
x_i(t)=\varphi_i(t), \quad  t \leq t_1,  \quad  i=1,\dots,n,
\end{equation}
where $a_{ij}(t)$ is equal to zero in the absence of the corresponding term in the system.
Similarly to the vector case, we assume that for the parameters of \eqref{5.10}, \eqref{5.11} % and the initial function
conditions $(ii), (iii), (v)$ hold %.  Also assume that 
and that %the right-hand sides 
$f_i(t)\in {\mathbf L}_{\infty}[t_0,\infty)$. 
% are essentially bounded on $[t_0,\infty)$. % vector functions.
Denote 
$x(t)= (x_1(t),\dots,x_n(t))^T$,  $\Phi(t)=(\varphi_1(t),\dots,\varphi_n(t))^T$, 
$f(t)= (f_1(t),\dots,f_n(t))^T$, $A(t)=\{a_{ij}(t)\}_{i,j=1}^n$.
By definition, % of the fundamental function, %the  matrix function 
$X(t,s)=\{X_{ij}(t,s)\}_{i,j=1}^n, t\geq s\geq t_0$ is a  fundamental function of system \eqref{5.10} if it  solves %the matrix equation
$$
\left(\begin{array}{ccc}
\frac{\partial}{\partial t} X_{11}(t,s)&\dots& \frac{\partial}{\partial t}X_{1n}(t,s)\\
---&---&---\\
\frac{\partial}{\partial t} X_{n1}(t,s)&\dots& \frac{\partial}{\partial t}X_{nn}(t,s)
\end{array}\right)=
A(t) \left(\begin{array}{ccc}
 X_{11}(h_{11}(t),s)&\dots& X_{1n}(h_{1n}(t),s)\\
---&---&---\\
X_{n1}(h_{n1}(t),s)&\dots& X_{nn}(h_{nn}(t),s)
\end{array}\right)$$  for $t\geq s\geq t_0$, $X(s,s)=E$ and  $X(t,s)=0$ for $t<s$.

\begin{lemma} \cite{AS} \label{lemma5.2}
There exists a unique solution of IVP \eqref{5.10}-\eqref{5.11}, and for this solution the representation in
%formula 
\eqref{5.3} holds.
\end{lemma}

Denote $a_0(t)=\|A(t\|$,  $h_0(t)=\inf\limits_{1 \leq i,j \leq n}h_{ij}(t)$. %, $i,j=1, \dots, n$.

\begin{proposition}\label{th5.3}
Assume that at least one of the conditions a)-b) of Corollary \ref{c5.1} holds.
Then the fundamental function  of system \eqref{5.10}  exponentially decreases, and \eqref{5.10}   is ES.
If all  $t-h_{ij}(t)$, $i,j=1, \dots,n$ are bounded then  system \eqref{5.10} is UES. 
\end{proposition}

%The proof is similar to the proof of Theorem \ref{th5.1}.

\begin{remark}\label{remark5.2}
Proposition~\ref{th5.3} can be applied to a more general system than \eqref{5.10}
\begin{equation*}\label{5.10add}
\dot{x}_i(t)=\sum_{k=1}^m\sum_{j=1}^n a_{ij}^k(t)x_j(h_{ij}^k(t))+f_i(t), \quad t\geq t_0, \quad i=1,\dots,n.
\end{equation*}
%and apply to this equation .
\end{remark}

% We will give now more general stability theorem than 
Let us generalize Proposition~\ref{th5.3} in the spirit of Theorem~\ref{th4.5}. Let $\Omega^n=\{(i,j)\}_{i,j=1}^n$ be the set of all pairs of  indices $(i,j)$,
%To this end, 
%Denote 
$\Omega\subset \Omega^n$, $\Omega'=\Omega^n \backslash \Omega$,
$
\displaystyle h_0(t)=\inf_{(i,j)\in \Omega}h_{ij}(t)$,  $ \displaystyle  h(t)=\inf_{(i,j)\in \Omega'}h_{ij}(t)$,
$\displaystyle 
A(t)= \{ a_{ij}(t)\}_{(i,j)\in \Omega}$ ,  $\displaystyle a_0(t)=\|A(t)\|$,  $\displaystyle A'(t)= \{a_{ij}(t)\}_{(i,j)\in \Omega'}$, $a(t)=\|A'(t)\|.
$

\begin{proposition}\label{th5.4a}
Assume that all the conditions of Theorem~\ref{th3.1} hold for system \eqref{5.10}. %, where $A(t), a_0(t), a(t)$ are replaced with $B(t), b_0(t), b(t)$. 
Then the fundamental function  of equation \eqref{5.10} 
exponentially decreases and  \eqref{5.10} is ES. If all  the memories $t-h_{ij}(t)$, $i,j=1, \dots, n$
% $\tau_{ij}$ are finite 
are bounded then   \eqref{5.10}  is UES. 
\end{proposition}

%The proof follows from Theorem~\ref{th4.5}.

\subsection{FDEs with a distributed delay and integro-differential equations} 
\label{sec:5.2}

Consider  a  linear vector equation
\begin{equation}\label{5.15}
\dot{x}(t)=\sum_{k=1}^m \int_{h_k(t)}^t d_s G_k(t,s)x(s) + f(t), \quad t\geq t_0
\end{equation}
with initial condition \eqref{5.2}  for any $t_1\geq t_0$ ,
where for $G_k(t,s)$, $h_k(t)$ and  $\varphi(t)$ conditions
$(ii), (iv), (vi)$ hold, and $f(t)$ is a locally essentially bounded vector function on $[t_0,\infty)$. 
The same conditions will be assumed for all the models studied in Section~\ref{sec:5.2}.

A fundamental function $X(t,s)$ of  \eqref{5.15} is 
a solution for $t\geq s\geq t_0$ of the matrix equation
$$
\frac{\partial X}{\partial t} (t,s)=\sum_{k=1}^m ~ \int\limits_{\max\{s, h_k(t)\}}^t  \!\!\! \!\!\! d_{\xi}G_k(t,\xi)X(t,\xi)=0, \quad t\geq s,
\quad X(s,s)=E, \quad X(t,\xi)=0,~~ \xi<s.
$$

\begin{lemma}\cite[Theorem~B.6]{BBbook}\label{lemma5.3}
%There is a unique solution of 
IVP \eqref{5.15},\eqref{5.2} has a unique solution, %and  %for $t \geq t_1$  
it can be presented as %for this solution the representation 
\begin{equation}\label{5.16}
x(t)=X(t,t_1)x(t_1)+\int\limits_{t_1}^t X(t,s)f(s)ds
+\sum_{k=1}^m \int\limits_{t_1}^t X(t,s)\left(~ \int\limits_{\min\{t_1, h_k(s)\}}^{t_1} \!\!\!\! d_{\xi}G_k(s,\xi)\varphi(\xi)\right)ds.
\end{equation}
%holds for $t\geq t_1$.
\end{lemma}

\begin{remark}\label{r5.2a}
Formula \eqref{5.16} also follows from \eqref{2.20}.
\end{remark}

We assume that conditions of Lemma~\ref{lemma2.6} hold for all 
$\displaystyle (V_ky)(t)= \!\!\!\!\!\! \int\limits_{\max\{t_0, h_k(t)\}}^t \!\!\!\!\!\!\!\!\! d_s G_k(t,s)y(s)$, $k=1,\dots m$ ensuring that the operator 
$ V_k, k=1.\dots,m$ has a  $\Delta$-property, also  $(ii)$  implies (A)-condition  for all the delays $h_k(t)$.  
Equation \eqref{5.15} and all the models %with distributed delays 
considered in Section~\ref{sec:5.2}
 have a  $\Delta$-property, and  the delays satisfy (A)-condition. 

Equation \eqref{5.15} is equivalent to \eqref{4.13} with abstract Volterra operators, thus
all the relevant statements  % for equations with Volterra abstract operators 
are applicable to %equation 
\eqref{5.15}.
Therefore, as corollaries of Theorems \ref{th4.1}-\ref{th4.4}  for  equations with abstract Volterra operators, 
we have stability tests for equation \eqref{5.15}.

Introduce non-empty sets $\Omega\subset \{1,\dots,m\}$,  ${\Omega}'= \{1,\dots,m\} \backslash \Omega$.
Denote  $\displaystyle 
A_k(t)= \!\!\! \!\!\! \int\limits_{\max\{t_1,h_k(t)\}}^t \!\!\!\!\!\!\!  d_s G_k(t,s)$,  $\displaystyle A(t)=\sum_{k\in \Omega} A_k(t)$, 
$\displaystyle a_0(t)=\sum_{k\in \Omega}\| A_k(t)\|$, $k=1, \dots, m$,
$\displaystyle  a(t)=\sum_{k\in {\Omega}'}\| A_k(t)\|$,
$\displaystyle  h_0(t)=\min_{k\in \Omega} h_k(t)$,  $\displaystyle h(t)=\min_{k\in {\Omega}'} h_k(t)$,
where $\|A_k(t)\|=Var_{\xi \in [h_k(t),t]} G_k(t,\xi)$.  
  
\begin{proposition}
\label{th5.4} 
Assume that the %at least one of the 
conditions %a)-b) holds in 
of either Proposition~\ref{th5.1} or of Corollary~\ref{c5.1}  hold. %, where  $B_k, b_0, b$ are replaced  with $C_k(t), c_0,c$, respectively.
Then the fundamental function of  equation \eqref{5.15} exponentially decreases  and \eqref{5.15} is ES. 
If all $t-h_k(t)$, $k=1, \dots, m$ %the delays 
are bounded  then \eqref{5.15} is UES.
\end{proposition}

%\begin{corollary}\label{c5.2}
%Assume that  the %at least one of  
%conditions % a)-b) 
%of Corollary~\ref{c5.1} hold. %,  where all $B_k(t)$ are replaced with $C_k(t)$. 
%Then the fundamental function of  equation \eqref{5.15} exponentially decreases  and \eqref{5.15} is ES. If in addition all the delays are bounded then \eqref{5.15} is UES.
%\end{corollary}

Consider IVP for a vector  equation with a non-delay term
\begin{equation}\label{5.21}
\dot{x}(t)=A(t)x(t)+\sum_{k=1}^m \int_{h_k(t)}^t d_s G_k(t,s)x(s)+f(t), ~t\geq t_0
\end{equation}
and %for any $t_1\geq t_0$ 
initial condition \eqref{5.2}. 
Denote  for some $t_1\geq t_0$, $\displaystyle a_k(t)= \!\!\!\!\!\!\!\!\! \int\limits_{\max\{t_1,h_k(t)\}}^t \!\!\!\!\!\! d_s G_k(t,s)$, $\displaystyle a_0(t) = \|A(t)\|$, $\displaystyle a(t)=\sum_{k=1}^m \|A_k(t)\|$, 
$h_0(t)=t$, $h(t)=\min_{k} h_k(t)$.

\begin{proposition}
\label{th5.5} 
Assume that %  at least one of the 
the conditions of %a)-b) holds in 
Proposition~\ref{th5.1} hold. %, where $B_k, b_0, b$ are replaced  with $C_k(t), c_0,c$. %, respectively.
Then the fundamental function of  equation \eqref{5.21} exponentially decreases  and \eqref{5.21} is ES. If in addition all
$t-h_k(t)$, $k=1, \dots, m$ are bounded then \eqref{5.21} is UES.
\end{proposition}

Consider now %for any $t_1\geq t_0$ 
a system of scalar FDEs  with distributed delays and initial condition \eqref{5.11}
\begin{equation}\label{5.22}
\dot{x}_i(t)=\sum_{j=1}^n \int_{h_{ij}(t)}^t x_j(s) d_s g_{ij}(t,s)  +f_i(t), \quad t\geq t_0,   \quad i=1,\dots,n.
\end{equation}
%and for any $t_1\geq t_0$ initial condition \eqref{5.11}.

\begin{proposition}\label{th5.6} 
Assume that for some $t_1\geq t_0$, all the conditions in Proposition~\ref{th5.3} hold, where  % $b_{ij}(t)$  are replaced  with 
%$\displaystyle a_{ij}(t)=  \!\!\!\!\!\!\!\!\! \!\!\!\!   \int\limits_{\max\{t_1,h_{ij}(t)\}}^t \!\!\!\!  \!\!\!\! \!\!\!\! d_s g_{ij}(t,s)$.
$\displaystyle a_{ij}(t)=  {\mathrm Var}_{\zeta \in [ \max\{t_1,h_{ij}(t)\}, t]}    g_{ij}(t,\zeta)$.
Then the fundamental function of  equation \eqref{5.22} exponentially decreases  and \eqref{5.22}   is ES. 
If all  $t-h_{if}(t)$, $i,j=1, \dots, n$ %the delays 
are bounded then \eqref{5.22}  is UES.
\end{proposition}

Denote $\Omega\subset \Omega^n$,  where $\Omega^n=\{(ij)\}_{i,j=1}^n$ is the set of all the pairs of indices $(i,j)$, $\Omega'=\Omega^n \backslash \Omega$,
$\displaystyle 
A(t)=\{a_{ij}(t)\}_{(i,j)\in \Omega}$, $a_0(t)=\| A(t)\|$, $\displaystyle A_1(t)=\{a_{ij}(t)\}_{(i,j)\in \Omega'}$,
$ a(t)=\|A_1(t)\|$,
$\displaystyle  h_0(t)=\min_{(i,j)\in \Omega}h_{ij}(t)$,  $\displaystyle h_1(t)=\min_{(i,j)\in \Omega'}h_{ij}(t)$,
where  %we set 
$a_{ij}(t)\equiv 0$ in the absence of the corresponding term in the system. 
Theorem~\ref{th4.5} implies

\begin{proposition}\label{th5.6a}
Assume that all the conditions of Theorem~\ref{th3.1} hold for system \eqref{5.22}. %, where $A(t)$, $a_0(t)$, $a(t)$ are replaced with $C(t), c_0(t), c(t)$. 
Then the fundamental function $X(t,s)$ of %equation 
\eqref{5.22} 
exponentially decreases and \eqref{5.22}  is ES. If all  $t-h_{ij}(t)$, $i,j=1, \dots,n$ %the memories $\tau_{ij}$ 
are bounded then  \eqref{5.22} is UES. 
\end{proposition}

%The result follows from Theorem~\ref{th4.5}.

For  $G_j(t,s)=\int_{t_0}^s K_j(t,\xi) d\xi$, equation~\eqref{5.15} with distributed delays transforms into
an integro-differential equation
\begin{equation}\label{5.23}
\dot{x}(t)=\sum_{j=1}^m \int_{h_j(t)}^t K_j(t,s)x(s)ds + f(t),\quad t\geq % t_1\geq 
t_0 .
\end{equation}

Let either the conditions of Lemma~\ref{lemma2.5} for all the operators 
$(V_ky)(t)=\int_{\max\{t_0,h_j(t)\}}^t K_j(t,s)x(s)ds$,  $j=1,\dots,m$ or the conditions of Lemma~\ref{lemma2.7} if $h_k(t)=t_0$
 hold.
Then the operators  $V_k$, $k=1,\dots,m$ have a $\Delta$-property.,  $(ii)$ implies (A)-condition.
%Since  $(ii)$ implies (A)-condition for all the delays,  all the stability tests for abstract 
%equations with Volterra linear operators apply to  \eqref{5.23} and all
%other integro-differential equations considered in Section~\ref{sec:5.2}. 
Denote for some $t_1\geq t_0$,
$\displaystyle a_j(t)= \!\! \! \int\limits_{\max\{t_1,h_j(t)\}}^t   \!\! \!  \|  K_j(t,s)\| \, ds$. As a corollary of Proposition~\ref{th5.4}, we get  %the following statement.

\begin{proposition}\label{th5.7} 
Assume that the conditions of  Proposition~\ref{th5.4} hold. %, where $c_j(t)$ are replaced with $d_j(t)$.
Then the conclusions of  Proposition~\ref{th5.4}  are valid for equation \eqref{5.23}.
\end{proposition}

Consider now % for any $t_1\geq t_0$ 
a vector equation including a non-delay term
\begin{equation}\label{5.23a}
\dot{x}(t)=A(t)x(t)+\sum_{j=1}^m \int_{h_j(t)}^t K_j(t,s)x(s)ds + f(t), \quad t %\geq t_1
\geq t_0.
\end{equation}
%with initial condition \eqref{5.11} for any $t_1\geq t_0$.

\begin{proposition}\label{th5.8} 
Assume that the conditions of  Proposition~\ref{th5.5} hold. %, where $c_j(t)$ are replaced  with $d_j(t)$.
Then the conclusions of  Proposition~\ref{th5.5}  are true for equation \eqref{5.23a}.
\end{proposition}

Consider now  a system of scalar linear integro-differential equations
\begin{equation}\label{5.23b}
\dot{x_i}(t)=\sum_{j=1}^n \int_{h_{ij}(t)}^t K_{ij}(t,s) x_j(s)ds +f_i(t), \quad t\geq t_0, \quad  i=1,\dots,n.
\end{equation}
Denote for some $t_1\geq t_0$,  $\displaystyle a_{ij}(t)=\int\limits_{\max\{t_1,h_{ij}(t)\}}^t  | K_{ij}(t,s) | ds$.

\begin{proposition}\label{th5.9} 
Assume that the conditions of Proposition~\ref{th5.6} hold. %, where $C_{ij}(t)$ are replaced with  $D_{ij}(t)$.
Then the conclusions of  Proposition~\ref{th5.6} are valid for equation \eqref{5.23b}.
\end{proposition}

Together with  \eqref{5.15} and \eqref{5.23}, consider  % for any $t_1\geq t_0$ 
equations with  multiple distributed delays: %Volterra integral operators can be considered:
\begin{equation}\label{5.24}
\dot{x}(t)=\sum_{k=1}^m \int_{t_0}^t d_s G_k(t,s)x(s) + f(t),~t\geq t_0, 
\end{equation}
%and
\begin{equation}\label{5.25}
\dot{x}(t)=\sum_{j=1}^m \int_{t_0}^t K_j(t,s)x(s)ds+f(t), ~t\geq t_0.
\end{equation}
 
Propositions~\ref{th5.4} and \ref{th5.7} can be applied to equations \eqref{5.24} and \eqref{5.25}, where only ES can be considered
for equations of this type and not UES. We omit details of such stability tests for  \eqref{5.24} and \eqref{5.25} and  for partial cases of these equations.

\subsection{Mixed FDEs with both concentrated and distributed delays}
\label{sec:5.3}

Equations of a mixed type contain terms with both a concentrated delay such as $B(t)x(h(t))$ and  a distributed delay 
$\int_{g(t)}^t d_s G(t,s)x(s)ds$. Let $ (i)-(iv)$  %hold and also 
and conditions of 
Lemmas~\ref{lemma2.5}--\ref{lemma2.7} be satisfied. 
Then all the equations of mixed types enjoy  (A)-condition  and $\Delta$-property. 
 For mixed equations,  there are many partial cases,  and in each case,
a general stability test could be constructed. 
Instead, we consider % for any $t_1\geq t_0$,
as an example, a system with two terms including different types of delay
% the following vector equation with two terms. In the next section we consider 
%a system of two scalar equations of the mixed type.
\begin{equation}\label{6.1}
\dot{x}(t)=A(t)x(h_1(t))+\int_{h_2(t)}^t d_s G(t,s)x(s), \quad t\geq t_0
\end{equation}
Denote $a_0(t)=\|A(t)\|$, $h_0(t)=\min\{h_1(t),h_2(t)\}$,
$\displaystyle 
A_1(t)= \!\!\! \!\!\!\!\!\!  \!\!   \int\limits_{\max\{t_1,h_2(t)\}}^t  \!\!  \!\!\! \!\!\!\! d_s G(t,s), ~a(t)=Var_{h_2(t)\leq s\leq t} G(t,s)$. 
Then Theorems~\ref{th3.1}, \ref{c4.1} and Corollaries~\ref{cor_th3.1}, \ref{c4.1a} imply

\begin{proposition}\label{th6.1} 
Let  for some $t_1\geq t_0$ and $\alpha \in (0,1)$ at least one of the conditions hold:
\\
a) %ODE 
%\begin{equation}\label{6.2}
$\dot{x}(t)=\mu_A(t) x(t)$
%\end{equation} 
is UES and % \\
%\begin{equation}\label{6.3}
$\displaystyle \sup_{t\geq t_1}\int_{t_1}^t e^{\int_s^t \mu_A (\xi)d\xi}\left\{ a_0(s)  \!\!
 \!\!\!\!  \!\!\!\! \int\limits_{\max\{t_1,h_1(s)\}}^s \! \!\!\!\!  \!\!\!\! \left[ a_0(\xi)+a(\xi)\right]   d\xi+a(s)\right\}ds\leq \alpha$;
%\end{equation}
\\
b) % ODE 
%\begin{equation}\label{6.2a}
$\dot{x}(t)=\mu_{A_1}(t) x(t)$
%\end{equation} 
is UES and  
$\displaystyle \sup_{t\geq t_1}\int_{t_1}^t e^{\int_s^t \mu_{A_1} (\xi)d\xi}\left\{ a(s)   \!\! \!\!\!\!\!
\int\limits_{\max\{t_1,h_2(s)\}}^s    \!\!   \!\! \!\!\!\! \left[ a_0(\xi)+a(\xi)\right] \, d\xi+a_0(s)\right\}ds$;
\\
c) %ODE 
%\begin{equation}\label{6.4a}
$\dot{x}(t)=\mu_{A+A_1}(t) x(t)$
%\end{equation} 
is UES and \\
%\begin{equation}\label{6.5a}
$\displaystyle \sup_{t\geq t_1}    \int\limits_{t_1}^t e^{\int_s^t \mu_{A+A_1} (\xi)d\xi}\left\{ (a_0(s)+a(s))   \!\!\!
\int\limits_{\max\{t_1,h_0(s)\}}^s   \!\!   \!\!  \left[ a_0(\xi)+a(\xi)\right]  d\xi\right\}ds\leq \alpha$.

%\end{equation}
%hold.
Then the fundamental function  of equation \eqref{6.1} exponentially decreases and \eqref{6.1}  is ES. If  $t-h_k(t)\leq \tau$, $k=1,2$ for some $\tau>0$
then \eqref{6.1} is UES.
\end{proposition}

%The proof follows from Theorems~\ref{th3.1}, \ref{c4.1} and Corollaries~\ref{cor_th3.1}, \ref{c4.1a}. 

\section{Examples}
\label{sec:6}

The examples we have collected in this section illustrate the features of our article that distinguish it from many other publication on stability of FDEs. Based on just one %theorem (
Theorem~\ref{th3.1}, we obtained results for many classes of differential equations which are usually considered separately.
In particular, Example~\ref{ex7.1}  deals with a scalar delay differential equation, where both delays are concentrated,  Example~\ref{ex7.6}  with a scalar differential equation including a distributed delay. 
Examples~\ref{ex7.2} - \ref{ex7.4}  explore  vector DDEs. In Examples~\ref{ex7.5}  and   \ref{ex7.7},   systems of scalar FDE are considered.
In Example~\ref{ex7.7}, the system contains terms with both concentrated and distributed delays.
Examples illustrate the cases  when the leading term  has either an oscillatory coefficient 
(Examples~\ref{ex7.1} and \ref{ex7.5}) or an oscillatory  kernel (Example~\ref{ex7.6}).
We are not aware of any previously known %stability conditions 
UES for systems with  oscillatory  kernels and measurable parameters. 
%in the latter case may have been obtained for the first time. 
Examples~\ref{ex7.4} and \ref{ex7.7}  include an unbounded delay. Asymptotic stability tests for such systems are known, 
but we obtained stronger UES conditions, allowing preservation of ES under  perturbations, 
in contrast to asymptotic stability. The properties of  UES preservation under small perturbations are illustrated 
in  Example~\ref{ex7.3}.  
%We emphasize that we do not assume the 
Continuity of parameters, unlike most other works on stability of FDE, is not assumed in the current paper.
In particular, in Example~\ref{ex7.4} the delays are discontinuous. 

\medskip

First, consider a scalar delay differential equation
with two delays, where both coefficients and their sum are oscillatory functions.

\begin{example}\label{ex7.1}
Consider %for any $t_1\geq t_0$ 
a scalar delay equation with two oscillatory coefficients
\begin{equation}\label{5.8}
\dot{x}(t)=(-0.5+\sin (10 t))x(h_0(t)) + 0.1 \cos^3 (t)  x(h(t)), \quad t\geq 0,
\end{equation}
%with initial condition \eqref{5.2},
where $0\leq t-h_0(t)\leq \tau_0$. Denote
$(L_0x)(t)=(-0.5+\sin 10 t)x(h_0(t))$, $(Lx)(t)=0.1 \cos^3 (t) x(h(t))$,  %$h_0(t)=h_1(t)$, $h(t)=h_2(t)$,
$A(t)=\mu_A(t)=-0.5+\sin (10 t)$,
$a_0(t)=|A(t)|\leq 1.5$, $a(t)=|0.1 \cos^3 (t)|\leq 0.1$ and check the conditions of Part a) of  Proposition \ref{th5.1} including  \eqref{5.7c}.

First,  
let us check that % the equation 
$\dot{x}(t)+\mu_A(t)x(t)=0$ is UES.
%We have  
%$X_{B}(t,s)= e^{\int_s^t \mu_B(\xi) d\xi}$, and we 
We estimate its fundamental function
$$
 X_A(t,s) =e^{\int_s^t \mu_A (\xi)d\xi}= e^{\int_s^t( -0.5+\sin (10 \xi)) d\xi}\leq
 e^{0.2} e^{-0.5(t-s)}.
$$ 
Hence  ODE $\dot{x}(t)=\mu_A (t) x(t)$  is UES.
Next, let us check  inequality \eqref{5.7c} with $t_1=0$, i.e.  %in Part a) of Proposition~\ref{th5.1} 
\begin{equation}\label{5.9a}
\sup_{t\geq 0}\left\{ a_0(t)\int_{h_0(t)}^t  \left[ a_0(\xi)+a(\xi) \right] \, d\xi
+a(t)\right\}~
\sup_{t\geq 0}\int_{0}^t e^{\int_s^t \mu_A (\xi)d\xi}ds
\leq \alpha<1.
\end{equation}
%in Theorem \ref{th3.1} holds.
 %As 
$\displaystyle 
\sup_{t\geq 0}\int_{0}^t e^{\int_s^t \mu_A (\xi)d\xi}ds \leq \sup_{t\geq 0}\int_0^t e^{0.2} e^{-0.5(t-s)}ds =  2e^{0.2}
$ and 
$\displaystyle 
\sup_{t\geq 0}\left\{ a_0(t)\int\limits_{h_0(t)}^t  \left[ a_0(\xi)+a(\xi) \right] \, d\xi
+a(t)\right\}$ \\ $\displaystyle \leq 1.5\tau_0(1.5+0.1)+0.1=2.4 \tau_0+0.1$
together imply inequality \eqref{5.9a}  for  $2e^{0.2} (2.4\tau_0+0.1)<1$.  
Hence equation \eqref{5.8} is UES for  $\tau_0<0.1289$ and any $\tau<\infty$, where $h_0$, $h$ can be discontinuous.
\end{example}

Next, take % equation we obtain UES stability conditions for 
a vector delay differential equation with a larger number of terms.

\begin{example}\label{ex7.2}
Consider a vector delay equation
\begin{equation}\label{5.9}
\dot{x}(t)=\sum_{k=1}^{10} A_k(t)x(h_k(t)), \quad t\geq 0, \quad A_k(t)=\left(\begin{array}{cc}
-1.1&0.1\sin(kt)\\0.1\cos(kt)&-1.1\end{array}\right),~~t\geq 0,
\end{equation}
where
$\displaystyle
 h_k(t)=t-0.05 \cos^2 (kt),~k=1,\dots,10.
$
Let us check %Corollary~\ref{c5.2} , checking 
the condition $\mu_A(t)\leq \mu_0<0$:  $\displaystyle
A(t)=\sum_{k=1}^{10} A_k(t)= \left(\begin{array}{cc}
-11&0.1\sum_{k=1}^{10}\sin(kt)\\0.1\sum_{k=1}^{10}\cos(kt)&-11\end{array}\right)$, thus  $\mu_A(t)\leq -10 = \mu_0 <0$, $t\geq 0$.
Next, check  $\displaystyle   \sup_{t\geq 0}  a_0(t) \!\! \int\limits_{h_0(t)}^t \!\! a_0(\xi)d\xi< |\mu_0| $ in  Corollary~\ref{c5.1}, Part b).
As 
$h_0(t)=\min_k h_k(t)$, $t-h_0(t)\leq 0.05$ and  %$ |\mu_B(t)| \geq  10$, 
$a_0(t)=\|A(t)\|\leq 12$, we have
$\displaystyle
a_0(t)\int_{h_0(t)}^t  a_0(s)ds \leq  12^2\cdot  0.05=7.2<|\mu_0|=10.
$
By Corollary~\ref{c5.1}, the fundamental function of equation \eqref{5.9} exponentially decreases  and \eqref{5.9}  is ES.
As $0\leq t-h(t)\leq 0.05$, \eqref{5.9} is UES.
\end{example}

Example \ref{ex7.3} demonstrates that addition of an oscillatory term in \eqref{5.9} can keep its  stability.

\begin{example}\label{ex7.3}
Consider a vector delay equation
\begin{equation}\label{5.9b}
\dot{x}(t)=\frac{\sin(t^2)}{1+t^2} A_0 x(h_0(t))+\sum_{k=1}^{10} A_k(t)x(h_k(t)),~t\geq 0,
\end{equation}
where $A_0$ is any constant matrix, $0\leq t-h_0(t)\leq \tau$  for some $\tau>0$  and $A_k,h_k$ are the same as in Example~\ref{ex7.2}.
We have  $\displaystyle \int_0^{\infty}\frac{|\sin(s^2)|}{1+s^2}ds\leq
\int_0^{\infty}\frac{1}{1+s^2}ds=\frac{\pi}{2} <\infty$. As was shown in 
Example~\ref{ex7.2}, equation \eqref{5.9} is UES. By Theorem~\ref{th4.1}, Part (2), equation \eqref{5.9b} is also  UES.
\end{example}

Next, take a vector equation with an unbounded discontinuous  delay.

\begin{example}\label{ex7.4}
Consider a vector delay equation with $x(t) \in {\mathbb R}^n$
\begin{equation}\label{5.10b}
\dot{x}(t)=A(t)x(t)+B(t)x(h(t)),~t\geq 1,
\end{equation}
where % $n\times n$ matrix 
$\displaystyle 
A(t)=\left(\begin{array}{cccc}
-n&\sin (t)&\dots&\sin (t)\\
\sin (2t)&-n&\dots&\sin (2t)\\
\dots&\dots&\dots&\dots\\
\sin (nt)&\sin (nt)&\dots &-n\end{array}
\right),
$
$B(t)=\frac{\beta \cos (t^2)}{t+1} E$, $ h(t)=t-\lfloor\ln t \rfloor$
%, where $\lfloor y \rfloor$ is the maximum integer not exceeding $y$,
and $\lim\limits_{t\rightarrow\infty} (t-\lfloor\ln t \rfloor) = \lim\limits_{t\rightarrow\infty}  (t-\ln t)=+\infty$. 
Thus $h(t)$ has (A)-property but
the memory  $\tau(t)=t-h(t)=\lfloor \ln t \rfloor$ in unbounded. 
Introduce the operator 
$\displaystyle 
(Vy)(t)=\left\{\begin{array}{ll}
B(t)y(t-\lfloor\ln t \rfloor ),&t-\lfloor\ln t \rfloor \geq 1,\\
0,&t-\lfloor\ln t \rfloor <1 \end{array}\right.
$
and check that it has a $\Delta$-property. For this, it is sufficient  to verify % the inequality 
$\displaystyle \sup_{t\geq 1}\|B(t)\|e^{\lambda (t-h(t))}<\infty$ for some $\lambda>0$.
Take $\lambda=1$, then
$\displaystyle
\left| \beta\frac{|\cos (t^2)|}{t+1}e^{\lfloor\ln t \rfloor  }  \right| \leq |  \beta | \frac{|\cos (t^2)|}{t+1}e^{\ln t  } \leq \frac{ |\beta|   t}{t+1}\leq |\beta| <\infty.
$
Hence the operator $V$ has a $\Delta$-property. In addition, $\mu_A (t)\leq -1<0$,  $b(t)=\|B(t)\|\leq \frac{|\beta|}{t+1}<\varepsilon$ for any small $\varepsilon \in (0,1)$  and  $t\geq t_1\geq t_0$ for $t_1(\varepsilon)$ large enough.
By the second inequality in Part b) in Proposition~\ref{th5.2}, % for $\varepsilon<1$ 
the fundamental function of equation \eqref{5.10b}  exponentially decreases, and \eqref{5.10b} is ES for any $\beta$.
%, the fundamental function of equation \eqref{5.10b}  exponentially decreases and \eqref{5.10b}  is ES.
Compared to \cite{BTT}, we can state exponential, not asymptotic stability, compared to \cite{Faria-2022},  we  consider a discontinuous delay, hence the stability tests in \cite{Faria-2022}  do not apply.
\end{example}

In Example~\ref{ex7.5}, a system with an oscillatory matrix measure %of two scalar delay differential equations
 is studied.

\begin{example}\label{ex7.5}
Consider a  system
\begin{equation}\label{5.14}
\begin{array}{l}
\displaystyle
\dot{x}_1(t)=-\alpha \cos^2 (nt)x_1(h_{11}(t))+\beta \sin^2 (nt) x_2(h_{12}(t)), \\
\displaystyle \dot{x}_2(t)=\beta \sin^2 (nt) x_1(h_{21}(t))-\alpha \cos^2 (nt)x_2(h_{22}(t)), \end{array} \quad  t\geq 0,
\end{equation}
where $\alpha >|\beta|$, $0\leq t-h_{ij}(t)\leq \tau$.  
Let us imtroduce 
$$
 A(t) = \left(\begin{array}{cc}
-\alpha \cos^2 (nt)&\beta \sin^2 (nt) \\
\beta \sin^2 (nt) &-\alpha \cos^2 (nt)\end{array}\right), \quad  a_0(t)=\|A(t)\|=\alpha \cos^2 (nt)+|\beta| \sin^2 (nt)\leq \alpha, \quad   \mu_A(t)
$$$$
%\mu_A(t)
=-\alpha \cos^2 (nt)+|\beta| \sin^2 (nt)
=-\alpha\frac{1+\cos(2nt)}{2}+|\beta|\frac{1-\cos(2nt)}{2}=
\frac{-\alpha+|\beta|}{2}-\frac{\alpha+|\beta|}{2}\cos(2nt)
$$
and
%Let us 
check inequality \eqref{2.2b} for $a(t)=\mu_A(t)$ and some $H>0$.
We have  
\begin{align*}
\int_t^{t+H} \mu_A (s)ds  & = \frac{-\alpha+|\beta|}{2} H-
\frac{\alpha+|\beta|}{4n}\big[ \sin (2n(t+H))-\sin(2nt) \big]
%\\
%& 
\leq - \frac{\alpha - |\beta|}{2}H +\frac{\alpha+|\beta|}{2n}.
\end{align*}
Therefore
$
H> \frac{\alpha+|\beta|}{n (\alpha-|\beta|)}
$
implies  \eqref{2.2b}, thus  ODE $\dot{x}(t)=\mu_A(t)x(t)$  is UES. 
To apply Corollary~\ref{c5.1} with \eqref{5.6b}  in~Part a), we further check inequality \eqref{5.6b}:
\\
$\displaystyle 
 e^{\int_s^t \mu_{A} (\xi)d\xi}=
 e^{\int_s^t \left[\frac{-\alpha+|\beta|}{2}-\frac{\alpha+|\beta|}{2}\cos(2n\xi)\right]d\xi}
\leq e^{ \frac{-\alpha+|\beta|}{2}(t-s)} \,  e^{\frac{\alpha+|\beta|}{2n}},$ $\displaystyle 
\int_{0}^t e^{\int_s^t \mu_{A} (\xi)d\xi}ds\leq \frac{2}{\alpha-|\beta|}e^{\frac{\alpha+|\beta|}{2n}},
$ \\ $\displaystyle 
a_0(t)\int\limits_{h_0(t)}^t a_0(\xi)d\xi\leq  \alpha^2 \tau.
$
Hence,  \eqref{5.6b} holds and therefore  \eqref{5.14} is UES  whenever
$\displaystyle
\frac{2 \alpha^2 }{\alpha-|\beta|}e^{\frac{\alpha+|\beta|}{2n}}\tau<1
$,
 or $\displaystyle \tau < \frac{\alpha-|\beta|}{2\alpha^2} e^{ - \frac{\alpha+|\beta|}{2n}}$.
For example, if $\alpha=2, \beta=1, n=5 $, %system 
\eqref{5.14} is UES if $\tau< \frac{1}{8}e^{-0.3}\approx 0.0926$.
\end{example} 

A linear integro-differential equation with an oscillatory kernel is investigated in Example ~\ref{ex7.6}.

\begin{example}\label{ex7.6}
Consider an integro-differential equation with an oscillatory kernel
\begin{equation}\label{5.26}
\dot{x}(t)=\int_{t-0.1}^t e^{-10(t-s)}\big[ -0.8+\sin(10\pi s) \big] \, x(s)ds, \quad  t\geq 0.
\end{equation}
We will apply Part a) of Corollary~\ref{c4.1a}.  First, compute the matrix measure
$$A(t)=\mu_A(t)  =\int_{t-0.1}^t e^{-10(t-s)}[-0.8+\sin(10\pi s)]ds
=-0.08(1-e^{-1})+I(t) < -0.05 + |I(t)|,$$
%where 
$\displaystyle | I(t)|= \left| \frac{1-e^{-1}}{10(1+\pi^2)}\sin(10\pi t)-\frac{\pi(1+e^{-1})}{10(1+\pi^2)}\cos(10\pi t) \right|
%$ and 
%$\displaystyle  
%|I(t)|
\leq \frac{\sqrt{(1-e^{-1})^2+\pi^2(1+e^{-1})^2}}{10(1+\pi^2)}\leq 0.04$.
Hence  $\displaystyle  \mu_A(t)\leq -0.05+|I(t)|\leq -0.01<0$,
$\displaystyle 
 a_0(t)\leq 0.08(1-e^{-1})+|I(t)|\leq 0.091$,  $\displaystyle \int_0^t e^{\int_s^t \mu_A(\xi)d\xi}ds\leq 100
$
and $100 \cdot (0.091)^2\cdot 0.1<0.082<1$, thus equation \eqref{5.26} is UES.
\end{example}

Finally, let us illustrate the case of  both concentrated and  distributed unbounded delays. 

\begin{example}\label{ex7.7}
Consider  a system
\begin{equation}\label{6.6}
\begin{array}{ll}
\dot{x}(t) & \displaystyle =-3 x(t-0.05 \sin^2 t)+ 0.5\int_{t/2}^t e^{-(t-s)} y(s)~ds, \\
\dot{y}(t) & \displaystyle = 0.95 x(t- 0.05\cos^2 t)- 3 y(t-0.03) -  \int_{t-0.05}^t e^{-(t-s)} y(s)~ds, 
\end{array}
\quad t\geq 0
\end{equation}
with an unbounded delay in the first equation (no finite memory).  All  the delays in the system have 
(A)-property. %Let us check that the operator with unbounded delay
The operator $\displaystyle  (Vy)(t)=\int_{t/2}^t e^{-(t-s)} y(s)~ds$  with an unbounded delay
satisfies
$\displaystyle \sup_{t\geq 0}\int_{t/2}^t e^{-(t-s)} e^{\lambda(t-s)}~ds<\infty
$ for any $\lambda<1$, and by \eqref{2.12} in Lemma~\ref{lemma2.4}, %the operator 
$V$ has a $\Delta$-property.

Let us apply here Theorem~\ref{th3.1}.
Denote 
\begin{align*}
(L_{11}x)(t)=-3 x(t-0.05 \sin^2 t), & ~~ (L_{12}y)(t)=0.5\int_{t/2}^t e^{-(t-s)} y(s)~ds,
\\
(L_{21}x)(t)=0.95 x(t- 0.05\cos^2 t), & ~~ (L_{22}y)(t)=- 3 y(t-0.03) -  \int_{t-0.05}^t e^{-(t-s)} y(s)~ds.
\end{align*}
Let $\Omega=\{(1,1),(2,2)\}$, $\Omega'=\{(1,2),(2,1) \}$, $X=(x,y)^T$. Hence
$\displaystyle
(L_0 X)(t)=(L_{11} x, L_{22}y)^T$,  
$\displaystyle
(L X)(t)=(L_{12}y,L_{21}x)^T$.
Then
$\displaystyle
A(t)=(L_0 E)(t)=\left(\begin{array}{cc}
L_{11}1&0\\0&L_{22}1\end{array}\right)
=\left(\begin{array}{cc}
-3&0\\0&-4+e^{-0.05 t}\end{array}\right),
$
$\displaystyle
\mu_A(t)\leq -3,~  a_0(t)=\|A(t)\|=4-e^{-0.05 t}\leq 4, ~ h_0(t)=t-0.05.
$
Also,
$
\displaystyle \|(L X)(t)\|\leq  a(t)  \sup_{t/2\leq s\leq t}\|x(s)\|$, where $a(t)\leq \max\{0.5,0.95 \} = 0.95.$
Let us  check 
\begin{align*}
\alpha(t)= & \int_0^t e^{\int_s^t \mu_A(\xi)d\xi}\left[a_0(s)\int_{h_0(s)}^s (a_0(\xi)+a(\xi))d \xi+a(s)\right]ds\leq \alpha<1:
\\
\alpha(t)\leq & \int_0^t e^{-3(t-s)}\left[4\int_{s-0.05}^s (4+0.95)d\xi +0.95\right]ds \leq
\frac{1}{3}(4 \cdot 4.95 \cdot 0.05+0.95) = \frac{1.94}{3} \approx 0.6467 <1.
\end{align*}
Hence the fundamental system of \eqref{6.6} exponentially decreases, and the system is ES.
\end{example}

\section{Conclusions}
\label{sec:7}

The paper develops explicit ES and  UES conditions  for a broad class of linear FDEs  formulated via abstract Volterra operators.
This framework unifies and generalizes existing stability results for many classes of equations, including DDEs, IDEs, equations with distributed delays,
vector equations  and  systems of scalar FDEs.

%\begin{enumerate}
%\item
%delay differential equations (DDEs),
%\item
%Integro-differential equations (IDEs),
%\item
%equations with distributed delays,
%\item
%vector equations  and  systems of scalar FDEs.
%\end{enumerate}

The central contribution is  {\em general stability theorem} (Theorem~\ref{th3.1}), from which explicit tests for all these types of equations follow as  corollaries. The approach relies on the Bohl-Perron theorem  for bounded Volterra operators with a $\Delta$-property, and it connects exponential stability with solvability of an operator equation in  ${\mathbf L}_{\infty} [t_0,\infty)$.
The theory is appropriate for equations with unbounded, time-varying, or oscillatory coefficients.

While investigation of exponential and global asymptotic stability is a topic of extensive research 
\cite{BB2006a,BB2024,BDSZ2,BDSZ1,Dom1,DomGus,G,GH2001,Kz,LizPituk,Malygina,Ngoc2016,Ngoc1,TuncTunc}, 
we emphasize the following novel elements  of Theorem~\ref{th3.1} and its corollaries:

\begin{enumerate}
\item
The method allows to establish UES property for systems with oscillatory coefficients or oscillatory kernels.
Most of presented examples involve oscillatory components, but in  Examples ~\ref{ex7.1} and  \ref{ex7.6}
the ``leading term'' used to establish  stability is oscillatory. We are not aware of similar results to compare these examples.
\item
When we add a perturbation term to a UES model, generally, the fact that it is everywhere  less than the leading  terms  in some sense is not assumed.
\item
Many results on global asymptotic stability and generalized exponential stability are strongly based on the continuity of parameters and the structure of the delay, 
%coefficients and delays, 
see
%, as examples, 
the papers \cite{GH2001,GP,Kz,Ngoc1, TuncTunc}. 
Our research uses more general framework and allows to establish ES and UES, where our knowledge about delays is only used in some standard operator estimations.
\end{enumerate}

In this paper we have not considered  linear  neutral differential
equations and non-linear FDE. We conjecture that the method of differential equations with abstract  Volterra operators
can also be applied to these classes of FDEs and propose it  as a topic for future research.
%In the paper we assume existence and uniqueness of solutions for IVP for 
%linear differential equations with abstract Volterra operators.
%It would be interesting to conduct an in-depth study of this problem,
%where some partial results  are available in  \cite{Cord}.
The main idea of the paper is to obtain a %the only one 
result for a suitable equation with abstract Volterra operators and then 
use it for many concrete FDEs. This approach can be applied to other qualitative problems of FDE, for example,
for boundary value problems.

\section*{Acknowledgment}

The research of both authors was  partially supported by the NSERC Discovery Grant \#  RGPIN-2026-04585.

\end{document}